 \documentclass{amsart}
\usepackage{graphicx} 
\usepackage{amsthm}
\usepackage{amsmath}
\usepackage{amssymb}
\usepackage{bbm}
\usepackage{cleveref}
\usepackage{eso-pic}
\usepackage[utf8]{inputenc}
\usepackage{tikz}
\usetikzlibrary{arrows.meta, positioning}
\usepackage{url}

\usepackage[top=2.74cm,bottom=2.74cm,left=2.5cm,right=2.5cm,marginparwidth=2cm, marginparsep=3mm]{geometry}

\newcommand{\A}{\mathbf{A}}
\newcommand{\B}{\mathbf{B}}
\newcommand{\G}{\mathbf{G}}
\newcommand{\I}{\mathbf{I}}
\newcommand{\Lg}{\mathbf{L}}
\newcommand{\Se}{\mathbf{S}}

\newcommand{\V}{\mathcal{V}}
\newcommand\SL{\mathcal{SL}}
\newcommand{\W}{\mathcal{W}}

\newcommand\PL{{\mathcal{P}}_{\textit{\l}}}

\newcommand\PLA{{\mathcal{P}_{\textit{\l}} (\mathbb{A})}}

\newtheorem{lemma}{Lemma}
\newtheorem{proposition}{Proposition}
\newtheorem{theorem}{Theorem}
\newtheorem{corollary}{Corollary}
\theoremstyle{definition}
\newtheorem{definition}{Definition}
\theoremstyle{remark}
\newtheorem{remark}{Remark}
\newtheorem{example}{Example}

\title{On some algebraic properties of P\l onka sums and regularized varieties}
\author{S. Bonzio and G. Zecchini}
\address{Dipartment of Mathematics and Computer Science \\ University of Cagliari, Italy.}
\email{stefano.bonzio@unica.it, giuseppe.zecchini1@gmail.com}

\date{}

\begin{document}

\begin{abstract}
\noindent P\l onka sums consist of a general construction that provides structural description for algebras in regularized varieties, whose examples range from Clifford semigroups to many algebras of logic including involutive bisemilattices, Bochvar algebras and certain residuated structures. While properties such as subdirectly irreducible algebras, subvariety lattices, and free algebras are well-understood for plural types without constants, the general case involving nullary operations remains largely unexplored. In this paper, we extend these results to algebraic types with constants and provide new insights into splittings within the lattice of subvarieties of a regularized variety. Furthermore, we offer a complete characterization of the congruences of a P\l onka sum and establish that the construction preserves surjective epimorphisms and injective monomorphisms.
\end{abstract}

\keywords{P\l onka sums, regularized varieties, lattice of subvarieties, splittings, free algebras, congruences, epimorphisms surjectivity.}
\subjclass[2020]{03C05, 08B15, 08B20.}

\maketitle

\section{Introduction}



The construction of P\l onka sums, introduced by Jerzy P\l onka in the late 1960s \cite{Plo67}, represents one of the most significant tools in universal algebra for decomposing algebras and characterizing varieties defined by regular identities. Historically, this construction found its roots in semigroup theory, where it was developed to represent Clifford semigroups. To semigroup theorists' the construction is widely known under the name of \emph{strong semilattice} of groups. In recent years, the relevance of Płonka sums has extended well beyond its original scope, finding significant applications in the study of non-classical logics, where they provide the natural algebraic framework for the logics of variable inclusion \cite{Bonziobook}. Furthermore, the construction has been increasingly employed beyond the theory of regular(ized) varieties, for which it turns out to be very useful \cite{Romanowska86},  reaching into the realm of residuated structures \cite{BalRes,bonzio2025balancedresiduatedpartiallyordered} and other classes of algebras arising from non-classical logics \cite{SMikBochvar} and connected to the logics of variable inclusion \cite{BonzioMorascoPrabaldi,Bonziobook,Bonziocontainment2}.

Given a variety $\V$, its regularization $R(\V)$ is defined as the variety satisfying the set of all regular identities (i.e. identities where the same variables appear on both sides) satisfied by $\V $. Under the mild additional assumptions that $\V$ is \emph{strongly irregular} (that is $\V\models x\cdot y\approx x$, for some term-definable operation $x\cdot y$), a property met by the majority of known irregular varieties, and of \emph{plural} type (i.e. the language of $\V$ contains at least an operation of arity strictly greater than one) it can be shown that every algebra in $R(\V)$ is (respresentable as) a P\l onka sum (over a semilattice direct system) of algebras is $\V$. Over the years, several fundamental algebraic properties of P\l onka sums and regularized varieties have been established. These include the characterization of subdirectly irreducible algebras \cite{Lakser72}, the description of the lattice of subvarieties of a regularized variety \cite{Dudek}, and the construction of free algebras \cite{Romanowskafree}. Furthermore, the work of Pastijn \cite{Pastijn} has provided crucial insights into the preservation of the Congruence Extension Property (CEP) and the Amalgamation Property (AP) through the construction of P\l onka sums.
However, a closer look to the existing literature reveals a significant restriction: despite the construction of the P\l onka sum has been extended to algebras (possibly) containing constants \cite{Plonka1984nullary}, the vast majority of these results have been proven exclusively for plural algebraic languages containing no constants. 
The aim of the present paper is to extend the structural theory of P\l onka sums and regular(ized) varieties in several new directions. The main contributions of this work are fourfold and, in particular, focus on:
\begin{enumerate}
    \item Extension to algebraic types possibly containing constants. We generalize the known results concerning subdirectly irreducible algebras, the lattice of subvarieties, and the description of free algebras to include algebraic types that may contain constants. 

\item Splittings in the lattice of subvarieties. We present entirely novel results regarding splitting pairs in the lattice of subvarieties of a regularized variety $R(\V)$ (for $\V$ being strongly irregular). Despite the intrisic value of this investigation for $R(\V)$, we also show that this has an impact also on splittings in the lattice of subvarieties of $\V$.

\item Description of congruences. We provide a complete description of the congruences of a P\l onka sum, solving one of the open problem listen in \cite{Bonziobook}. 

\item Preservation of (some) categorical properties. 
We establish that the properties of epimorphisms surjectivity and monomorphisms injectivity are preserved by the construction of the P\l onka sums, hence are transferred from a (strongly irregular) variety $\V$ to its regularization $R(\V)$. 
\end{enumerate}

The paper is organized as follows. In Section 2, we collect the necessary preliminaries concerning P\l onka sums and regular(ized) varieties, with a specific focus on the adjustments required when the algebraic type includes constants. Section 3 is devoted to the study of the lattice of subvarieties and the characterization of subdirectly irreducible algebras for regularized varieties. Here, we extend the classical results to the general case of algebras (possibly) containing nullary operations. In Section 4, we investigate splittings in the lattice of subvarieties of a regularized variety. Section 5 provides a comprehensive description of free algebras in regularized varieties. In particular, we show how the characterization by Romanowska \cite{Romanowskafree} can be extended to the case of (free) algebras that accounts for the presence of constants in the signature. In Section 6, we address the problem of characterizing the congruences of a P\l onka sum, by describing how a congruence of a P\l onka sums is built from the congruences of the summands and the underlying semilattice. As an application of our main result, we also include two subsections on generated congruences and factor congruences. Finally, in Section 7, we show that the properties of epimorphisms surjectivity and monomorphisms injectivity are preserved by the contruction of the P\l onka sum, hence transferred from a strongly irregular variety $\V$ to its regularization $R(\V)$.

\section{Preliminaries}\label{Sec: Preliminari}

We assume the reader has some familiarity with universal algebra (standard references are \cite{BergmanLibro,BuSa00}). 


Throughout the paper, we will assume  algebras to be defined in an arbitrary \emph{plural} (algebraic) language of type $\tau$ (thus containing at least one operation symbol of arity strictly greater than $1$) and possibly containing constant symbols, with no explicit mention to $ \tau$ (if not needed); we will explicitly mention the cases where the type is assumed not to contain constant symbols.   

A \emph{semilattice direct system} of similar algebras is a triple $\mathbb{A} = (\{\A_i\}_{i\in I}, (I, \le), \{p_{ij}\}_{i\le j})$ consisting of a (join) semilattice $\I = (I,\le)$ with least element $i_{0}$\footnote{The existence of a least element is necessary when working with types containing constants and can be relaxed in absence of constants.} (and with join $\vee$), a family of similar algebras $\{\A_i\}_{i\in I}$ with (pairwise) disjoint universes and a family of homomorphisms $\{p_{ij}\}_{i\le j}$, one for each $i\leq j\in I$, such that $p_{ii}$ is the identity and $p_{ik} = p_{jk}\circ p_{ij}$, whenever $i\le j \le k$. 

Given two semilattice direct systems (of algebras) $\mathbb{A} = (\{\A_i\}_{i\in I}, (I, \le), \{p_{ij}\}_{i\le j})$ and $\mathbb{B} = (\{\B_j\}_{j\in J}, (J, \le), \{q_{kl}\}_{k\le l})$, a morphism from $\mathbb{A}$ to $\mathbb{B}$ is a pair $\langle \varphi, \{f_i\}_{i\in I} \rangle$ such that $ \varphi\colon I\rightarrow J $ is a semilattice homomorphism and $ f_i\colon A_{i}\rightarrow B_{\varphi(i)} $ is a homomorphism for each $i\in I$, making the following diagram commutative for each $ i, j\in I $, $ i\leq^{I} j $.

\begin{figure}[h]
\begin{center}
\begin{tikzpicture}[scale=0.8]
\draw (-4,0) node {$ B_{\varphi(i)} $};
	\draw (4,0) node {$ B_{\varphi(j)} $};
	
	\draw [line width=0.8pt, ->] (-3.5,0) -- (3.4,0);
	\draw (0,-0.4) node {\begin{footnotesize}$q_{\varphi(i)\varphi(j)}$\end{footnotesize}};

	\draw [line width=0.8pt, <-] (3.3,3) -- (-3.3,3);
	\draw (0, 3.3) node {\begin{footnotesize}$p_{ij}$\end{footnotesize}};

	\draw (-4,3) node {$A_i$}; 
	
	\draw (4,3) node {$A_{j}$};
		
	\draw [line width=0.8pt, ->] (-4,2.6) -- (-4,0.4);
	\draw (-4.6,1.7) node {\begin{footnotesize}$f_i$\end{footnotesize}};
	\draw (4.6,1.7) node {\begin{footnotesize}$f_{j}$\end{footnotesize}};
	\draw [line width=0.8pt, <-] (4,0.4) -- (4,2.6);
	
\end{tikzpicture}

\end{center}
\end{figure}

With the above notion of morphism, semilattice direct systems form a category (see e.g. \cite{Bonziobook}). 

Given any semilattice direct system of algebras $\mathbb{A}= (\{\A_i\}_{i\in I}, (I, \le), \{p_{ij}\}_{i\le j})$, it is possible to define a new algebra $\A$ with universe the (disjoint) union $A = \bigcup_{i\in I} A_{i}$, whose $n$-ary operations (with $n\geq 1$) are defined as $g^{\A}(a_{1},\dots a_{n}):= g^{\A_{j}}(p_{i_{1}j}(a_{1}), \dots, p_{i_{n}j}(a_n))$, where $a_{1}\in A_{i_1}, \dots, a_{n}\in A_{i_n}$ and $j = i_{1}\vee\dots\vee i_{n}$; in case the type contains constants then each one of them is identified with their interpretation in $\A_{i_{0}}$, that is $c^{\A}:= c^{\A_{i_0}}$ (for every constant symbol $c\in \tau$). The new algebra $\A$, which we will sometimes indicate by $\PLA$, is called the \emph{P\l onka sum} of the algebras in the system $\mathbb{A}$, in honor of J. P\l onka who, inspired by the structure of strong semilattices of groups, firstly introduced this construction in Universal algebra \cite{Plo67, Plo67a, Plo68, Plonka1984nullary, RomanowskaPlonka92}.

Understanding whether a given arbitrary algebra $\A$ can be decomposed into a P\l onka sum over a certain semilattice direct system (of algebras) is connected with the existence of a special operation on $\A$, known as \emph{partition function}. In details, a function $\odot\colon A^2\to A$ is a \emph{partition function} in $\A$ if the following conditions are satisfied for all $a,b,c\in A$, $ a_1 , ..., a_n\in A $ and for any operation $g$ of arity $n\geqslant 1$ and any costant operation $c$.
\begin{enumerate}
\item[(PF1)] $a\odot a = a$,
\item[(PF2)] $a\odot (b\odot c) = (a\odot b) \odot c $,
\item[(PF3)] $a\odot (b\odot c) = a\odot (c\odot b)$,
\item[(PF4)] $g(a_1,\dots,a_n)\odot b = g(a_1\odot b,\dots, a_n\odot b)$,
\item[(PF5)] $b\odot g(a_1,\dots,a_n) = b\odot a_{1}\odot_{\dots}\odot a_n$, 
\item[(PF6)] $a\odot c^{\A} = a$.
\end{enumerate}

The connection between P\l onka sums and partition functions is stated in the following. 

\begin{theorem}[P\l onka decomposition theorem]\label{th: Plonka dec. theorem}
    Let $\A$ be an algebra of type $\tau$ with a partition function $\odot$. The following conditions hold: 
\begin{enumerate}
\item $A$ can be partitioned into $\{ A_{i} \}_{i \in I }$ where any two elements $a, b \in A$ belong to the same component $A_{i}$ exactly when
\[
a= a\odot b \text{ and }b = b\odot a.
\]
\item The relation $\leq$ on $I$ given by
\[
i \leq j \text{ \; if and only if there exist }a \in A_{i}, b \in A_{j} \text{ s.t. } b\odot a =b
\]
is a partial order and $\langle I, \leq \rangle$ is a semilattice. 
\item For all $i,j\in I$ such that $i\leq j$ and $b \in A_{j}$, the map $p_{ij} \colon A_{i}\to A_{j}$, defined by $x\mapsto p_{ij}:=(x)= x\odot b$ is a homomorphism. 
\item $\mathbb{A} = \langle \{ \A_{i} \}_{i \in I}, \langle I, \leq \rangle, \{ p_{ij} \! : \! i \leq j \}\rangle$ is a direct system of algebras such that $\PL(\mathbb{A})=\A$.
\end{enumerate}
\end{theorem}

We recall that, in case $\A$ is an algebra in a type not containing constant symbols, Theorem \ref{th: Plonka dec. theorem} can be reinforced adding that each algebra $\A_{i}$ (in the system $\mathbb{A}$) is a subalgebra of $\A=\PLA$ (see \cite{Plo67}). Moreover, \cite[Th. II]{Plo67} proves also that: (i) every representation of an algebra $\A$ as a P\l onka sum $\PLA$ is obtained by starting with a suitable partition function $\odot$ on $\A$; (ii): there is a bijective correspondence between partition functions (over a given algebra) and P\l onka sum representations. 

We will sometimes refer to the algebras $\A_{i}$ and the homomorphisms $\{p_{ij}: i\leq j\}$ in a semilattice direct systems (or a P\l onka sum) are \emph{fibers} and \emph{transition morphisms}, respectively.

Recall that a \emph{variety} is a class of algebras closed under the formation of subalgebras, homomorphic images and products; via the celebrated Birkhoff theorem, this is equivalent to be an equational class.   

An identity $\varphi \approx \psi$ (in an arbitrary algebraic language) is called \emph{regular} provided that the sets of variables occurring in $\varphi$ and $\psi$, respectively, are equal, that is $Var(\varphi)= Var(\psi)$. An identity which is not regular will be called \emph{irregular}. The construction of the (non-trivial) P\l onka sum $\PLA$ (over a semilattice direct system of algebras $\mathbb{A}$) satisfies all the regular identities holding in any algebra in $\mathbb{A}$ and falsifies any other identity (see \cite[Th. I]{Plo67}). A variety $\V$ is called \emph{regular} if it does satisfy regular identities only; a variety which is not regular will be called irregular. Examples of regular varieties include semigroups, monoids and semilattices (and many more as we shall see). Among the irregular varieties, we will particular focus on the ones defined in the following.

\begin{definition}
A variety $\V$ (of type $\tau$) is \emph{strongly irregular} if there exist a term-definable ($\tau$) formula $f(x,y)$ (where $x,y$ actually occur) such that $\V\models f(x,y)\approx x$.     
\end{definition}

Examples of strongly irregular varieties abound (while it is difficult to provide examples of an irregular variety which is not strongly irregular) and include the varieties of groups, rings, lattices, any variety with a lattice (or a group) reduct and also congruence modular varieties (see \cite{BergamRomanowskaQuasivarieties}). 

Given a strongly irregular variety $\V$, we will denote by $R(\V)$ the \emph{regularization} of $\V$, that is the
variety satisfying all and only the regular identities holding in $\V$. Regularizations $R(\V)$ are often called regularized varieties. Interestingly, regularized varieties and P\l onka sums are interconnected. 

\begin{theorem}\cite[Theorem 7.1]{RomanowskaPlonka92}\label{th: regularization of st.irregular varieties}
Let $\V$ be a strongly irregular variety. Then the following facts are equivalent:
\begin{enumerate}
\item $\A\in R(\V)$; 
\item $\A=\PLA$, with $\mathbb{A}$ a semilattice direct system of algebras in $\V$.
\end{enumerate}

\end{theorem}

The content of the previous theorem can be generalized to categories by saying that the (algebraic) category $R(\V)$ is equivalent to that of semilattice directed systems of algebras in $\V$ (see \cite[Ch. 3]{Bonziobook}).

Some regularized varieties have been extensively studied: some of them are recalled in the following. 
\begin{example}\label{ex: Clifford semigroups}
The variety of \emph{Clifford semigroups} is the regularization of the variety of groups (it follows from Theorem \ref{th: regularization of st.irregular varieties} and the well known structural characterization of Clifford semigroups as strong semilattices of groups, see e.g. \cite{CliffordPrestonBook, howie1995fundamentals}). In Clifford semigroups, the term-function $f(x,y):= x\cdot y\cdot y^{-1}$ plays the role of a partition function. 
\end{example}

\begin{example}\label{ex: IBSL}
The regularization of the variety of Boolean algebras (see e.g. \cite{Plonka1984nullary}) is also known as \emph{involutive bisemilattices} and has been extensively investigated in connection with (weak) Kleene logics (see \cite{Bonziobook, Bonzio16SL}). In involutive bisemilattices, the left-hand side of the absorption law, that is $f(x,y):=x \vee (x \wedge y)$ (or $f(x,y):= x \wedge (x \vee y)$), plays the role of a partition function.    
\end{example}

\begin{example}\label{ex: dual weak braces}
A skew brace $\G = (G, \cdot , \circ, ^{-1}, ^{'}, e )$ is an algebra of type $(2,2,1,1,0)$ such that the (term-)reducts $(G, \cdot, ^{-1}, e)$ and $(G, \circ, ^{'}, e)$ are groups and satisfy the (regular) identity $x \circ (y\cdot  z) \approx (x \circ y) \cdot x^{-1} \cdot (x \circ z)$. \\
\noindent 
Clearly, skew braces form a strongly irregular variety (having two group reducts). Its regularization has been introduced and studied recently in \cite{CatinoMediterranean} (see also \cite{Stefanelli25}) under the name of \emph{dual weak braces}, which are (essentially) defined as algebras $\mathbf{S} = (S, \cdot , \circ, ^{-1}, ^{'})$ of type $(2,2,1,1)$, where $(S, \cdot, ^{-1})$ and $(S, \circ, ^{'})$ are Clifford semigroups\footnote{The definition of dual weak braces is actually simpler and require only (the reduct) $(S, \cdot, ^{-1})$ to be Clifford and $(S, \circ, ^{'})$ only an inverse semigroup to get, in presence of the other additional identities, that also the latter is Clifford.} plus the identity defining skew braces paired with (the additional identity) $x\cdot x^{-1}\approx x\circ x'$. In dual weak braces, the terms $x\cdot y\cdot y^{-1}$ and $x\circ y\circ y'$ coincide, yielding the same partition function (see Example \ref{ex: Clifford semigroups}).   
\end{example}

\section{Subvarieties and subdirectly irreducible members of a regular(ized) variety} 

In this section, we show that the description of the lattice of subvarieties of a regularized variety (the regularization of a strongly irregular variety) provided in \cite{Dudek} applies also for varieties defined over languages contains constant symbols (see Theorem \ref{theorem: L}). In order to prove the main result, we first need to introduce a semilattice-style decomposition, introduced in \cite{Plonka1969} and more general than P\l onka sums, holding for the regularizations of irregular -- but not in general strongly irregular -- varieties (see Theorem \ref{theorem: irregulrep}). 

\begin{definition}
Let $\tau$ be a type of algebras with constants and let 
$\widetilde{\tau} $ denotes the set of operation symbols with non-null arity. The \emph{full non-nullary reduct} of a $\A$ is the algebra $\widetilde{\A}$ with same universe $A$ of $\A$ and operations defined by $ f^{\widetilde{\A}}=f^{\A}$, for every $f\in \widetilde{\tau}$. 
\end{definition}

\begin{definition}\label{def: semilattice of subalgebras}
   An algebra $\mathbf{A}$ is a \emph{semilattice of subalgebras} $\{\A_{i}\}_{i\in I}$ 
    if $(I,\le)$ is a semilattice with least element $i_{0}$ and $\{\mathbf{A}_i\}_{i\in I}$ is a family of pairwise disjoint algebras such that: 
    
    \begin{enumerate}
        \item $\displaystyle A=\bigcup_{i\in I}{A_i}$;
        \item $\mathbf{A}_{i_0} \le \mathbf{A}$ and 
$ \mathbf{A}_i \le \widetilde{\mathbf{A}}$, for every $ i\in I\setminus \{i_0\}$;
        \item for every operation $f$ or arity $n\geq 1$, $ f^{\mathbf{A}}(a_1,...,a_n)\in  \mathbf{A}_{i_1 \vee... \vee i_n} $, with $a_1\in A_{i_1},...,a_n \in A_{i_n} $ and for every constants $c$ in the type, $c^{\A}\in A_{i_{0}}$.
    \end{enumerate}
\end{definition}

Observe that every algebra which decomposes as a P\l onka sum (see Theorem \ref{th: Plonka dec. theorem}) is also a semilattice of subalgebras, while the converse does not hold in general. Nevertheless, semilattices of subalgebras satisfy some properties common with P\l onka sums. 

We will say that a semilattice of subalgebras $\{\A_{i}\}_{i\in I}$ is \emph{non-trivial} if the semilattice $(I,\leq)$ is non-trivial, i.e. contains more than one element.  
We omit the proofs of the following results as they work exactly as in the case of languages not containing constants.


\begin{lemma}\emph{\cite[Lemma 1]{Dudek}}\label{lemma: int}
   Let $\A$ be a semilattice of subalgebras $\{\A_{i}\}_{i\in I}$ and 
   $t=t(x_1,...,x_n)$ an $n$-ary term (in the algebraic language of $\A$).
   Then for every $ a_1 \in A_{i_1},...,a_n\in A_{i_n} $, $ t^{\mathbf{A}}(a_1,...,a_n) \in \mathbf{A}_i$, with $i = i_{1}\vee\dots\vee i_{n}$. 
\end{lemma}



\begin{lemma}\label{lemma: non-trivial semilattice}
A semilattice of subalgebras is non-trivial if and only if it satisfies only regular identities (in a fixed set of variables).
\end{lemma}


\begin{theorem}\emph{\cite[Theorem 4]{Plonka1984nullary}}\label{theorem: irregulrep}
    Let $\V$ be
    an irregular variety (of type $\tau$) and $\mathbf{A} \in R(\mathcal{V})$. Then $\mathbf{A}$ is a semilattice of subalgebras. 
    In particular, if the type $\tau$ contains constants then $\mathbf{A}_{i_0} \in \mathcal{V}$. 
\end{theorem}

Given a variety $\V$, we will denote by $\Lg(\V)$ the lattice of its subvarieties (by $L(\V)$ its universe) and by $\preccurlyeq $ the cover relation in $\Lg(\V)$.

\begin{remark}\label{rem: fatto ovvio su R}
For $\V$ a variety, we have that $\V\in\Lg(R(\V))$. Moreover, for $\V_{1},\V_{2}$ subvarieties of $\V$ such that $\V_{1}\subseteq\V_{2}$, it holds $R(\V_{1})\subseteq R(\V_{2})$. 
\end{remark}

The following result is proved essentially as  \cite[Theorem 1 \& Theorem 2]{Dudek} for the case without constants: we include it for the ease of readability. 

\begin{theorem}\label{theorem: cover}
Let $\V$ be a variety in a language (possibly) containing constant symbols. Then:   \begin{enumerate}
        \item if $\V$ is irregular, then for every $\mathcal{W} \in L(\mathcal{V})$, $ \mathcal{W} \preccurlyeq R(\mathcal{W})$ in $\Lg(R(\mathcal{V}))$;

        \item if $\mathcal{V}$ is strongly irregular, then the map $h$ defined by $\mathcal{W} \mapsto R(\mathcal{W})$ is an embedding of $\Lg(\mathcal{V})$ in $\Lg(R(\mathcal{V}))$.
    \end{enumerate}
    
\end{theorem}

\begin{proof} 
(1) Let $\mathcal{V}$ be irregular, so there exists an irregular identity $p \approx q$ (over some infinite set of variables $X$ in the language of $\V$) such that $\mathcal{V} \models p\approx q$, therefore for any $ \mathcal{W} \in L(\mathcal{V})$, $ \mathcal{W} \models p \approx q$, that is every $\mathcal{W} \in L(\mathcal{V})$ is also irregular and, consequently $ \mathcal{W} \subsetneq R(\mathcal{V})$. Let $\W' \in L(R(\mathcal{V}))$ such that $\mathcal{W} \subseteq \W' \subseteq R(\mathcal{W})$: we want to show that one of the inclusion is not proper. We distinguish two cases:
 \begin{itemize}
\item $\W'$ is regular, 
then, by Remark \ref{rem: fatto ovvio su R} $R(\W)\subseteq R(\W')\subseteq R(R(W)) = R(\W)$, therefore $\W' = R(\W') = R(\W)$.

\item $\W'$ is irregular. By \emph{\cite[Theorem 11.4]{BuSa00}} there exists a free algebra $\mathbf{U}$ such that $\W' = HSP(\mathbf{U})$, so $\mathbf{U} \in \W' \subseteq R(\W')\subseteq R(\mathcal{W})$, therefore $\mathbf{U} \in R(\mathcal{W})$. Since $\mathcal{W}$ is irregular, by Theorem \ref{theorem: irregulrep} we have that $\mathbf{U}$ is a semilattice of subalgebras 
$\{\mathbf{U}_{i}\}_{i\in I}$
such that $\mathbf{U}_{i_0} \in \mathcal{W}$ (with $i_{0}$ the least element in $\I$). But $\mathbf{U \in \W'}$, so $\mathbf{U}$ satisfy at least one irregular identity, therefore by Lemma \ref{lemma: non-trivial semilattice} we have $|I|=1$, that is $\mathbf{U}=\mathbf{U}_{i_0} \in \mathcal{W}$. Consequently $\W'=HSP(\mathbf{U})\subseteq \mathcal{W}$ and $\W'=\mathcal{W}$.
        \end{itemize}
\noindent
(2) To see that $h$ is injective, suppose that we have $R(\mathcal{V}_1)=R(\mathcal{V}_2)$, for some $\mathcal{V}_1,\mathcal{V}_2 \in L(\mathcal{V})$  Reasoning as above, there exists a free algebra $\mathbf{U}\in \mathcal{V}_1$ such that $\mathcal{V}_1=HSP(\mathbf{U})$. Consequently $\mathbf{U} \in R(\mathcal{V}_1)= R(\mathcal{V}_2)$, so $\mathbf{U} \in R(\mathcal{V}_2)$. Now, the same reasoning applied in the proof of (1) shows that $\mathbf{U} \in \mathcal{V}_2$, hence $\mathcal{V}_1=HSP(\mathbf{U}) \subseteq \mathcal{V}_2$. $\mathcal{V}_2\subseteq \mathcal{V}_1$ is proven analogously. \\
\noindent
To show that $h$ is a lattice homomorphism, let $\mathcal{V}_1,\mathcal{V}_2\in L(\mathcal{V})$. We have $\mathcal{V}_1 \vee \mathcal{V}_2 \subseteq R(\mathcal{V}_1) \vee R(\mathcal{V}_2) \subseteq R(\mathcal{V}_1 \vee \mathcal{V}_2)$; moreover, by \cite[Corollary 2]{Plonka1984nullary}, $\mathcal{V}_1 \vee \mathcal{V}_2$ is irregular, therefore by (1) $R(\mathcal{V}_1)\vee R(\mathcal{V}_2)=R(\mathcal{V}_1 \vee \mathcal{V}_2)$ (since $R(\mathcal{V}_1)\vee R(\mathcal{V}_2)$ is regular). This proves that $h$ is join-preserving.

\noindent
Finally, using the P\l onka sums decomposition (see Theorem \ref{th: Plonka dec. theorem}) of $\mathbf{A} \in R(\mathcal{V}_1) \cap R(\mathcal{V}_2)$, it is easily shown that $\A \in R(\mathcal{V}_1 \cap \mathcal{V}_2)$, so $ R(\mathcal{V}_1)\cap R(\mathcal{V}_2) \subseteq R(\mathcal{V}_1 \cap \mathcal{V}_2)$, while the converse inclusion is immediate.
 \end{proof}

Theorem \ref{theorem: cover} allows to provide an equational basis to any irregular variety from one of its regularization.


\begin{corollary}\label{cor: equational basis of V}
Let $\V$ be an irregular variety (in a language possibly containing constants), $\Sigma$ an equational basis for $R(\V)$ and $p\approx q$ any irregular identity holding in $\V$. Then $\Sigma \cup \{p\approx q\}$  is an equational basis for $\V$. 
\end{corollary}

In the following examples, we present alternative equational basis for the varieties of Boolean algebras, groups and skew braces, respectively, using equational basis of their respective regularizations (see Examples \ref{ex: Clifford semigroups}, \ref{ex: IBSL} and \ref{ex: dual weak braces}). 
\begin{example}\label{ex: Boolean algebras}
The following is an equational basis (see \cite{Bonzio16SL}) for involutive bisemilattices (see Example \ref{ex: IBSL}) in the (algebraic) language $\vee,\wedge, \neg, 0,1$ (of type $(2,2,1,0,0)$).   
\begin{enumerate}
\item $x\lor x\approx x$;
\item $x\lor y\approx y\lor x$;
\item $x\lor(y\lor z)\approx(x\lor y)\lor z$;
\item $\neg\neg x\approx x$;
\item $x\land y\approx\neg(\neg x\lor\neg y)$;
\item $x\land(\neg x\lor y)\approx x\land y$; 
\item $0\lor x\approx x$;
\item $1\approx\neg 0$.
\end{enumerate}
Boolean algebras can thus be axiomatized by adding e.g. the single identity $x\approx x\vee (x\wedge y)$ to (1)-(8). 
\end{example}

\begin{example}\label{ex: gruppi}
Clifford semigroups (see e.g. \cite{howie1995fundamentals}) can be axiomatized as follows (in the language $\cdot, ^{-1}$): 
\begin{itemize}
\item[(CS1)] $x\cdot (y\cdot z) \approx (x\cdot y)\cdot z$; 
\item[(CS2)] $(x^{-1})^{-1} \approx x$;
\item[(CS3)] $x\cdot (x^{-1}\cdot x) \approx x $; 
\item[(CS4)] $(x\cdot x^{-1})\cdot (y\cdot y^{-1})\approx (y\cdot y^{-1})\cdot (x\cdot x^{-1})$;
\item[(CS5)] $x\cdot x^{-1}\approx x^{-1}\cdot x$.
\end{itemize}
It follows from Corollary \ref{cor: equational basis of V} that groups can be axiomatized just by adding the identity $x\cdot (y\cdot y^{-1})\approx x$ to (CS1)-(CS5).
\end{example}

\begin{example}\label{ex: skew braces}
Dual weak braces, introduced in Example \ref{ex: dual weak braces}, are axiomatized by the following identities (see e.g. \cite{Stefanelli25}): 
\begin{itemize}
\item[(DWB1)] $x\cdot (y\cdot z) \approx (x\cdot y)\cdot z$; 
\item[(DWB2)] $(x^{-1})^{-1} \approx x$;
\item[(DWB3)] $x\cdot (x^{-1}\cdot x) \approx x $; 
\item[(DWB4)] $(x\cdot x^{-1})\cdot (y\cdot y^{-1})\approx (y\cdot y^{-1})\cdot (x\cdot x^{-1})$;
\item[(DWB5)] $x\cdot x^{-1}\approx x^{-1}\cdot x$.
\item[(DWB6)] $x\circ (y\circ z) \approx (x\circ y)\circ z$; 
\item[(DWB7)] $(x')' \approx x$;
\item[(DWB8)] $x\circ (x'\circ x) \approx x $; 
\item[(DWB9)] $(x\circ x')\circ (y\circ y')\approx (y\circ y')\circ (x\circ x')$;
\item[(DWB10)] $x \circ (y\cdot  z) \approx (x \circ y) \cdot x^{-1} \cdot (x \circ z)$;
\item[(DWB11)] $x\cdot x^{-1}\approx x\circ x'$.
\end{itemize}
It follows by Corollary \ref{cor: equational basis of V} that the above identities plus the single identity $x\cdot (y\cdot y^{-1}) \approx x$ is an equational basis for (the variety of) skew braces.

\end{example}

 In the following we denote by $\mathbf{2}$ the $2$-element (chain) semilattice (with universe $\{0,1\}$).

\begin{theorem}\label{theorem: L}
    Let $\tau$ be an algebraic language containing constants and $\mathcal{V}$ a strongly irregular $\tau$-variety, then $\mathbf{L}(R(\mathcal{V}))) \cong \mathbf{L}(\mathcal{V}) \times \mathbf{2}$.
\end{theorem}

\begin{proof}
The proof runs exactly as that of \cite[Theorem 3]{Dudek}: we just give a sketch of it. 
Let $\phi\colon \Lg(\mathcal{V}) \times \mathbf{2} \rightarrow \Lg(R(\mathcal{V}))$ be defined as follows:

\begin{center}
    
    $
    \phi(\mathcal{W},j):=\begin{cases}
    \mathcal{W}, \text{ for } $j=0$;
    \\
    R(\mathcal{W}), \text{ for } $j=1$.
    \end{cases}$
\end{center}

It is routine to check that $\phi$ is an injective homomorphism. To prove surjectivity, observe that if $\mathcal{W} \in L(\V) $ then $\phi((\mathcal{W},0))=\mathcal{W}$. So suppose $\mathcal{W} \in L(R(\mathcal{V})) \setminus L(\mathcal{V})$. Let $\mathbf{U}$ be a free algebra such that $\mathcal{W}=HSP(\mathbf{U})$. Since $\mathcal{V}$ is strongly irregular and $\mathbf{U} \in HSP(\mathbf{U})=\mathcal{W} \subseteq R(\mathcal{V})$, by Theorem \ref{th: regularization of st.irregular varieties} we have that $\mathbf{U}$ is a P\l onka sum over a semilattice direct system $\mathbb{U}=(\{\mathbf{U}_i\}_{i\in I},(I, \le), \{p_{ij}\}_{i\le j})$ of algebras in $\mathcal{V}$. Let $X$ be an infinite set of variables and denote by $Id_{\A}(X)$ and $\mathcal{R}(Id_{\A}(X))$, respectively, the set of identities in variables from $X$ satisfied by an algebra $\A$ and the subset of the regular ones, respectively. Then by the theory of P\l onka sums, we have $\displaystyle \mathcal{R}(Id_{\mathbf{U}}(X))=Id_{\mathbf{U}}(X)=\bigcap_{i\in I}{\mathcal{R}(Id_{\mathbf{U}_i}(X))}=\mathcal{R}(Id_{\prod_{i\in I}{\mathbf{U}_i}}(X))$. Since $\mathcal{V}$ is a variety, we have $\displaystyle \prod_{i\in I}{\mathbf{U}_i} \in \mathcal{V}$. Let $\displaystyle \mathcal{K}=HSP\left(\prod_{i\in I}{\mathbf{U}_i} \right)$, then $\mathcal{K} \in L(\mathcal{V})$ and $Id_{\mathcal{R}(\mathcal{K})}(X)=\mathcal{R}(Id_{\mathcal{K}}(X))=\mathcal{R}(Id_{\prod_{i\in I}{\mathbf{U}_i}}(X))=Id_{\mathbf{U}}(X)=Id_{\mathcal{W}}(X)$, so $\mathcal{W}=\phi((\mathcal{K},1))$.
    \end{proof}

We will denote by $\mathcal{T}$ the trivial variety, i.e. the variety satisfying any identity (in a fixed algebraic language). It can be shown that the regularization $R(\mathcal{T})$ of the trivial variety is the variety of $\mathcal{SL}^{\tau} $ that is term-equivalent to the variety of semilattices (see e.g. \cite{RomanowskaPlonka92} for details). Moreover, a variety $\V$ is regular if and only if contains $\mathcal{SL}^{\tau}$ as subvariety\footnote{The assumption that the algebraic type contains at least an operation of arity greater than 1 is crucial for this result.} \cite[Proposition 2.1]{RomanowskaPlonka92}. With a slight notational abuse, from now on we will indicate the variety $\mathcal{SL}^{\tau}$ simply as $\mathcal{SL}$.



 \begin{corollary}
    Let $\tau$ be an algebraic plural language containing constants and $\V$ an irregular variety. Then $R(\V)=\V \vee \mathcal{SL}$.
\end{corollary}

\subsection{Subdirectly irreducible algebras}


The characterization of subdirectly irreducible P\l onka sums given in \cite{Lakser72} actually holds also for algebras containing constant symbols. We recall the main results for which we provide only sketches of the proofs as they run exactly as in \cite{Lakser72}. 

\begin{definition}\label{def: Astar}
    Let $\A$ be an algebra and $\infty \not \in A$ the universe of a trivial algebra. We denote by $\A^{\ast}$ the P\l onka sum over the semilattice direct system $(\{\A_i=\A, \A_j=\{\infty\}\}, \{i<j\}, \{p_{ij}\})$.
\end{definition}


\begin{lemma}\label{lemma: SI1}
    Let $\A$ be a P\l onka sum over a semilattice direct system of algebras indexed on $\I=(I,\leq)$. 
    If $\A$ is subdirectly irreducible, then $\mathbf{I}$ is subdirectly irreducible.
\end{lemma}

\begin{proof}
The proof strategy runs as follows: 
 \begin{enumerate}
 \item to each $\theta\in Con(I)$ associate a congruence $\theta^{\A}\in Con(\A)$ defined as 
 
 $$\text{for } a\in A_i, b\in A_j: (a,b) \in \theta^{\A} \iff \exists k\in I: i,j\le k, (i,k),(j,k) \in \theta  \text{ and } p_{ik}(a)=p_{jk}(b).$$
 
 \item Show that, for every $ \theta \in Con(I) $ if $ \theta^{\A}=\Delta_{\A}$ then $\theta = \Delta_{\mathbf{I}}$.

\item Show that, given $D\subseteq Con(\mathbf{I})$ a family such that $\displaystyle \bigcap D = \Delta_{\mathbf{I}}$ then $\Delta_{\mathbf{I}} \in D$. The argument proceeds by contradiction, considering the family $D^{\A} = \{\theta^{\A}:\theta \in D\}\subseteq Con(\A)$, $D_1=\{\psi \in Con(\mathbf{I}): \psi \neq \Delta_{\mathbf{I}}, \exists\theta \in D: \psi \subseteq \theta\}$, $D_2=D \cup D_1$ and showing that $\displaystyle \bigcap D_2^{\A}=\Delta_{\A}$. 

         \end{enumerate}
\end{proof}

\begin{definition}
    An element $a$ of a non-trivial algebra $\A$ is an \emph{absorbing element} if, for every operation symbol $f$ (of arity $n\geq 1$), $a\in \{x_1,...,x_n\}$ implies that $f^{\A}(x_1,...,x_n)=a$. 
\end{definition}

\begin{theorem}\emph{\cite[Main Theorem]{Lakser72}}\label{theorem: SI2}
    Let $\A$ be a P\l onka sum over a semilattice direct system $\mathbb{A} = (\{\A_i\}_{i\in I}, (I, \le), \{p_{ij}\}_{i\le j})$ of algebras in a variety $\V$ (possibly containing constant symbols). Then $\A$ is subdirectly irreducibile if and only if one of the following occurs:

    \begin{enumerate}
        \item $\A$ is subdirectly irreducible in $\V$;
        \item There exists $\B \in \V$ subdirectly irreducible and without absorbing elements such that $\A \cong \B^{\ast}$.
    \end{enumerate}
\end{theorem}

\begin{proof}

By Lemma \ref{lemma: SI1}, $\mathbf{I}$ is subdirectly irreducibile, that is it is trivial or it is it (up to isomorphisms) the two-element chain 
    $i<j$. 
    In the former case we immediately get (1), in the latter define $\theta_1, \theta_2 \in Con(\A)$ as:
   $$\forall x,y\in A: (x,y)\in \theta_1 \iff x=y \text{ or } x,y\in A_j ; $$
 $$\forall x,y\in A: (x,y) \in \theta_2 \iff \exists h\in \{i,j\}: x,y\in A_h \text{ and } p_{hj}(x)=p_{hj}(y). $$

 \noindent 
 $\theta_1 \cap \theta_2=\Delta$, so $\theta_1=\Delta$ or $\theta_2=\Delta$ (as $\A$ is subdirectly irreducibile) and, in both cases,$|A_j|=1$, so $\A \cong \A_i^{\ast}$. Finally, one shows (by contadiction) that $\A_i$ has no absorbing element and that $\A_i$ is subdirectly irreducible.


\end{proof}

\begin{corollary}\label{corollary: irreducible1}\emph{\cite[Corollary]{Lakser72}}
    Let $\V$ be a strongly irregular variety (possibly containing constant symbols). Then $\A\in R(\V)$ is subdirectly irreducible if and only if one of the following occurs:

    \begin{enumerate}
        \item $\A$ is subdirectly irreducible in $\V$;
        \item There exists a subdirectly irreducible algebra $\B\in \V$ such that $\A=\B^{\ast}$.
    \end{enumerate}
\end{corollary}


\begin{corollary}
    Let $\V$ be a strongly irregular variety. Then $\A$ is simple in $R(\V)$ if and only if $\A$ is simple in $\V$ or $\A$ is the two-element semilattice. 
\end{corollary}

\begin{proof}
Let $\A\in R(\V)$ be simple, then $\A$ is also subdirectly irreducible and by Corollary \ref{corollary: irreducible1} $\A$ is subdirectly irreducible, in particular simple, in $\V$ or there exists a subdirectly irreducible member $\B$ of $\V$ such that $\A=\B^{\ast}$. Observe that $\B$ can not be non-trivial, as otherwise the congruence $\nabla^{\B} \cup \{(\infty,\infty)\}\not \in \{\Delta^{\B^{\ast}}, \nabla^{\B^{\ast}}\}$, against the assumption that $\A$ is simple. Consequently, $\B$ is trivial and $\A$ is the two-element semilattice.
The converse is immediate.
\end{proof}

\section{Splittings in regular varieties}

\noindent The characterization of the lattice of subvarieties of a regularized varieties turns out to be very fruitful to describe \emph{splittings} in such a lattice, which is the content of this section.

Given a a lattice $\mathbf{L} = \langle L,\wedge,\vee\rangle$, we indicate by $\uparrow a$ and $\downarrow a$ (respectively) the upset and downset of $a$ (respectively) of an element $a$, that is $\uparrow a = \{x\in L\; |\; a\leq x\}$, $\downarrow a = \{x\in L\;|\; x\leq a\}$.
 Recall a pair of elements $a,b\in L$ (with $a\nleq b$) splits $\mathbf{L}$, if for every $ c\in L$, $a\leq c$ or $c\leq b$; in such a case $(a,b)$ is said to be a \emph{splitting pair} and $L = \uparrow a\;\cup\downarrow b$. It readily follows that $a$ is a join-irreducible element and $b$ is a meet-irreducible element. In case $\mathbf{L} = L(\V)$ is the lattice of subvarieties of a given variety $\V$, then, for every splitting pair $(\V_1 , \V_2)$, it holds that: 
\begin{itemize}
\item $\V_1 = V(\A)$, with $\A$ a subdirectly irreducible finitely generated algebra in $\V$ ($\A$ is called \emph{splitting algebra}); 
\item $\V_2$ is axiomatized (with respect to $\V$) by a single identity, called \emph{splitting identity}. 
\end{itemize}

\noindent Splitting algebras (in a variety) are connected with projective members\footnote{Sometimes they are also called \emph{weakly projective} (see e.g. \cite{HardingRomanowskaPart1}).} of a variety (see e.g. \cite{JipsenRose}). In particular, the following is a useful corollary of general result by A. Day (\cite{Day75}, \cite[Theorem 2.11]{JipsenRose}).
 
\begin{corollary}\label{cor: projective si are splitting}
Every projective subdirectly irreducible algebra $\mathbf{P}$ in a variety $\V$ is splitting in $\mathbf{L}(\V)$.
\end{corollary}

Let us now introduce one of the main examples of splitting pairs in regular varieties.

\begin{example}\label{ex: SL and V}
For a strongly irregular variety $\V$, the pair ($\SL$, $\V$) is splitting in the lattice $\Lg(R(\V))$ (it follows directly from Theorem \ref{theorem: L}). Accordingly, $\SL$ is generated by the two-element semilattice and $\V$ is axiomatized by a single identity of the form $x\cdot y \approx x$, for some operation $x\cdot y$, playing the role of partition function in $R(\V)$ and witnessing the strong irregularity of $\V$ (see e.g. \cite{RomanowskaPlonka92}).
\end{example}

When $\V$ is a strongly irregular variety, $\mathbf{L}(R(\V))$ is essentially made of two isomorphic copies of $\Lg(\V)$ (by Theorem \ref{theorem: L}), i.e. $\Lg(R(\V)) = \Lg(\V) \cup \Lg_R$, where $\Lg_R$ is the isomorphic copy of $\Lg(\V)$ consisting of all the regularizations of the varieties in $L(\V)$, thus having $R(\V)$ and $\SL$ as top and minimum element, respectively. In particular, $\Lg(\V)\cong \Lg_R$, via the map $\Phi$ sending $\Phi(\mathcal{W})= R(\mathcal{W})$. These observations imply the following.


\begin{lemma}\label{lem: induced splitting}
Let $\V$ be a strongly irregular variety and $\V_1, \V_2 \in L(\V)$. Then the following are equivalent: 

\begin{itemize}
    \item[(i)] $(\V_1, \V_2)$ is a splitting pair in $\Lg(\V)$;
    \item[(ii)] $(R(\V_1), R(\V_2))$ is a splitting pair in $\Lg_R$;
    \item[(iii)] $(\V_1, R(\V_2))$ is a splitting pair in $\Lg(R(\V))$.
\end{itemize}
\end{lemma}

 Finding splitting identities is, in general, a difficult task, achieved sometimes with ad hoc methods (see e.g. \cite{Jankov} for Heyting algebras). However, looking to $R(\V)$ might be of some help, as specified by the following. 
\begin{theorem}\label{th: splitting id}
Let $\V$ be a strongly irregular variety. Then for every splitting pair $(\V_1 , \V_2)$ in $\Lg(\V)$ there exists an equational basis of $\V$ such that the splitting identity is regular.
\end{theorem}
\begin{proof}
By Corollary \ref{cor: equational basis of V}, an equational basis for $\V$ is given by $\Sigma\cup\{x\cdot y \approx x\}$, where $\Sigma$ is an equational basis of the regularization $R(\V)$ and $x\cdot y$ is a binary term in the language of $\V$. Let $(\V_1, \V_2)$ be a splitting pair in $\V$. Then, by Lemma \ref{lem: induced splitting}, $(R(\V_1), R(\V_2))$ is a splitting pair in $\Lg_R$ (the lattice of the regular subvarieties of $\Lg(R)$). This implies that there exist a single regular identity $\rho\approx \tau$ axiomatizing $R(\V_2)$ with respect to $R(\V)$. Again by Corollary \ref{cor: equational basis of V}, there exists a binary term $t(x,y)$ (in the language of $\V$) such that $\Sigma\cup \{\rho\approx\tau\} \cup \{t(x,y)\approx x\}$ is an equational basis for $\V_2$. On the other hand, since $(\V_1 , \V_2)$ is splitting in $\V$, $\V_2$ is axiomatized by a single identity, with respect to $\V$: we claim that $\rho\approx\tau$ is such identity. 
Observe that $\V_2\models x\cdot y \approx x$ (since $\V_2 \subset \V$) and $\V_2\models t(x,y)\approx x$, thus $\V_2\models x\cdot y\approx t(x,y)$. This implies that an equational basis for $\V_2$ is given by $\Sigma\cup\{x\cdot y \approx x\}\cup\{\rho\approx\tau\}$, i.e. $\V_2$ is axiomatized by $\rho\approx\tau$ with respect to $\V$, namely $\rho\approx\tau$ is the splitting identity for $(\V_1,\V_2)$.
\end{proof}


In order to find the regular splitting identity obtained in Theorem \ref{th: splitting id}, one has to necessarily look for a specific equational basis of a strongly irregular variety $\V$ (see Corollary \ref{cor: equational basis of V}) consisting of the following: 
\begin{itemize}
    \item any equational basis $\Sigma$ of $R(\V)$; 
    \item the single identity $x\cdot y\approx x$ witnessing the strongly irregularity of $\V$.
\end{itemize}

Despite this might lead to consider unsual equational basis for a strongly irregular variety $\V$, the process is not complicated (see Examples \ref{ex: Boolean algebras}, \ref{ex: gruppi}  and \ref{ex: skew braces}). From the proof of Theorem \ref{th: splitting id} we immediately get. 
\begin{corollary}
Let $\V$ be a strongly irregular variety. Then the splitting pairs $(\V_1 , \V_2)$ in $\V$ and $(\V_1 , R(\V_2))$ in $R(\V)$ have the same (regular) splitting identity.
\end{corollary}

\begin{theorem}
    Let $\V$ be a strongly irregular variety. Then $(\V_1, \V_2)$ is a splitting pair in $\Lg(R(\V))$ if and only if $(\V_1, \V_2)=(\mathcal{SL}, \V)$ or there exists $\W \in \Lg(\V)$ such that $\V_2=R(\W)$ and $(\V_1, \W)$ is a splitting pair in $\Lg(\V)$.
\end{theorem}

\begin{proof}
$(\Rightarrow)$. Let $(\V_1, \V_2)$ be a splitting pair in $\Lg(R(\V))$, that is $L(R(\V))=\uparrow \V_1\ \cup \downarrow \V_2$ with $\V_1 \not \subseteq \V_2$, i.e. $\uparrow \V_1\ \cap \downarrow \V_2 = \emptyset$. Let us distinguish two cases:\\
1) $\V_1$ is regular. Then we know that $\mathcal{SL} \subseteq \V_1$. Actually, it is the case that $\mathcal{SL}=\V$. Indeed, assume by contradiction that $\mathcal{SL} \subsetneq \V_1$ then, since the pair $(\V_1, \V_2)$ is splitting,  $\mathcal{SL} \in \downarrow \V_2$, that is $\V_2$ is also regular. Since $\V_1$ is regular we have (by Theorem \ref{theorem: L}) that $\V \not \in \uparrow \V_1$ (otherwise $\V$ would be regular), so $\V \in \downarrow \V_2$, and consequently $R(\V) \subseteq R(\V_2)=\V_2 \subseteq R(\V)$, therefore $\V_2=R(\V)$, which implies that $\V_1\cap\V_2 \neq \emptyset$, in contradiction with the hypothesis that $(\V_1, \V_2)$ is a splitting pair in $\Lg(R(\V))$. Therefore, $\mathcal{V}_1 = \mathcal{SL}$. Consequently, $\mathcal{SL} \not\in \downarrow \mathcal{V}_2$, which implies that $\mathcal{V}_2$ is irregular. On the other hand, $\mathcal{V} \not\in \uparrow \mathcal{V}_1$ (otherwise, $\mathcal{V}$ would be regular), hence $\mathcal{V} \in \downarrow \mathcal{V}_2$, so $\mathcal{V} \subseteq \mathcal{V}_2 \subseteq R(\mathcal{V})$, and thus, by Theorem \ref{theorem: cover}, $\mathcal{V}_2 = \mathcal{V}$. We have hence proved that $(\V_1, \V_2)=(\SL, \V)$.\\
\noindent
2) $\V_1$ is not regular. In this case, observe that $\mathcal{SL} \not\in \uparrow \mathcal{V}_1$. Indeed, if this were not the case, we would have $\mathcal{T} \subseteq \mathcal{V}_1 \subseteq \mathcal{SL}=R(\mathcal{T})$, and thus $\mathcal{V}_1=\mathcal{T}$, in contradiction with the hypothesis that $(\mathcal{V}_1, \mathcal{V}_2)$ is a splitting pair in $\Lg(R(\mathcal{V}))$. Consequently $\mathcal{SL} \in \downarrow \mathcal{V}_2$, hence $\mathcal{V}_2$ is regular. Then by Theorem \ref{theorem: L} there exists $\mathcal{W} \in L(\mathcal{V})$ such that $\mathcal{V}_2=R(\mathcal{W})$. Since by hypothesis $(\V_1, R(\W))$ is a splitting pair in $\Lg(R(\V))$ with $\V_1, \W \in \Lg(\V)$, by Lemma \ref{lem: induced splitting} we have that $(\V_1, \W)$ is a splitting pair in $\Lg(\V)$.

\noindent
$(\Leftarrow)$. It follows immediately from Lemma \ref{lem: induced splitting}.
\end{proof}

\noindent As an immediate consequence of the preceding theorem, we have the following results.

\begin{corollary}
    Let $\V$ be a strongly irregular variety. Then an algebra $\A$ is splitting in $\Lg(R(\V))$ if and only if it is splitting in $\Lg(\V)$.
\end{corollary}

\begin{corollary}
    Let $\A$ be a non-trivial subdirectly irreducible projective algebra in $\V$, then $\A$ is not projective in $R(\V)$.
\end{corollary}

\section{Free algebras}\label{sec: algebre libere}

The characterization of free algebras in the regularization of a strongly irregular variety  is provided by Romanowska \cite{Romanowskafree} (extending a previous result by P\l onka \cite{Plonka71}), under the (standard) assumption that the algebraic language is plural, thus, in particular contains no constant symbols. The technique in \cite{Romanowskafree} actually works also for algebraic languages containing constants, as we point out in this section.  


Let $\mathbf{A}\in R(\V)$ a free algebra on the set of generators $\{a_j\}_{j\in J}$. Let $x\cdot y$ be a binary term witnessing the strong irregularity of $\mathcal{V}$, that is $\mathcal{V}\models x\cdot y \approx x$, and $J_0=\{k_1,...,k_n\}\subseteq J$, then for every $i\in \{1,...,n\}$ we set     

\begin{center}
\begin{equation}\label{eq: generators}
g_i=a_{k_i}\cdot^{\mathbf{A}} a_{k_2}\cdot^{\mathbf{A}}...\cdot^{\mathbf{A}}a_{k_{i-1}}\cdot^{\mathbf{A}}a_{k_1}\cdot^{\mathbf{A}}a_{k_{i+1}}\cdot^{\mathbf{A}}...\cdot^{\mathbf{A}}a_{k_n}.
\end{equation} 
\end{center}
\noindent

Let $A_{\emptyset}=Sg^\mathbf{A}(\emptyset)$, $A_{J_0}=Sg^{\tilde{\mathbf{A}}}(g_1,...,g_n)$, where $S_g$ is the \emph{generated subuniverse (closure) operator} (see \cite[Definition 3.1]{BuSa00}), and $\mathbf{A}_{J_0}$ be the algebra with universe $A_{J_0}$ such that 
$f^{\A_{J_0}}=f^{\widetilde{\A}}|_{A_{J_0}}$, for every non-nullary operation $f$, and 
$c^{\mathbf{A}_{J_0}}= c^{\A}\cdot^{\mathbf{A}} g_j$, for every constant $c$ and $g_j$ an arbitrary element in $\{g_1,...,g_n\}$.  
Given a set $J$, by $\mathcal{P}_{fin}(J)$ we will mean the collection of all finite subsets of $J$, namely $\mathcal{P}_{fin}(J) = \{J_0 \subseteq J\;:\; J \text{ is finite} \}$.

\begin{lemma}\cite[Lemma 1]{Plonka71}\label{lemma: Lemma 1 libere}
For any $ J_0\in \mathcal{P}_{fin}(J)$, $ \A_{J_0}$ is a $\V$-free algebra.
\end{lemma}

\begin{proof}
It is not difficult to check that $ \A_{J_0}$ has the universal mapping property in $\V$.
\end{proof}


\begin{lemma}\cite[Lemma 2]{Plonka71}\label{lemma: Lemma 2 libere}
Let $\mathcal{V}$ be a strongly irregular variety and $\mathbf{A}$ be a free algebra in $R(\mathcal{V})$. Then $\mathbf{A}$ is a P\l onka sum over a semilattice direct system whose index semilattice is $(\mathcal{P}_{fin}(J),\subseteq)$ and whose fibers are $\{\mathbf{A}_{J_0}\}_{J_0\in \mathcal{P}_{fin}(J)}$.
\end{lemma}

\begin{proof}
Since $\mathbf{A}\in R(\V)$, by Theorem \ref{th: regularization of st.irregular varieties} $\A$ decomposes as a P\l onka sum over a semilattice direct system $\mathbb{B}=(\{\B_l\}_{l\in L},(L,\le^L),\{q_{u,v}\}_{u\le^L v})$ of algebras in $\mathcal{V}$. It is enough to check that the operation $\cdot^{\mathbf{A}}$ (which acts as a partition function on $\A$) provides exactly the P\l onka sum decomposition of $\A$ into $\mathbb{B}$, namely that two elements $a,b\in A$ belongs to $A_{J_0}$ (with $J_0\in \mathcal{P}_{fin}(J)$) iff $a\sim b$, i.e. $a\cdot^{\mathbf{A}} b = a$ and $b\cdot^{\mathbf{A}} a = b$. (see Theorem \ref{th: Plonka dec. theorem}). Observe that we are including the subset $J_0 = \emptyset$, since the constant operations of $\A$ belongs to $A_{\emptyset}$. 
\end{proof}

The following states an immediate consequence of the previous lemmas under the assumption that every finitely generated free algebra in $\V$ is finite. 

\begin{corollary}\label{corollary: Libere3}
Let 
$\A_n$ be the $n$-generated $R(\V)$-free algebra ($n\in \mathbb{N}$). Then

\begin{equation}\label{eq: elementi di una libera finitamente generata}
|A_n|=\sum_{j=0}^n{\binom{n}{j}}|B_j|,
\end{equation}
where $\B_j$ is the (finite) $\mathcal{V}$-free algebra on $j$ generators. In particular, if $\V$ is a strongly irregular locally finite variety then $R(\V)$ is locally finite.
\end{corollary}

\begin{proof}
    The first part of the corollary is an immediate consequence of lemmas \ref{lemma: Lemma 1 libere} and \ref{lemma: Lemma 2 libere}. As for the second part, by \cite[Theorem 10.15]{BuSa00} $R(\V)$ is locally finite if and only if each finitely generated $R(\V)$-free algebra is finite, but this follows immediately from the first part, since $\V$ is (by hypothesis) locally finite, so each of its finitely generated free algebra is locally finite.
\end{proof}





As a conseguence of Lemmas \ref{lemma: Lemma 1 libere} and \ref{lemma: Lemma 2 libere} we have that each free algebra in the regularization of a strongly irregular variety $\V$ in an algebraic language containing constants is a P\l onka sum over a semilattice direct system with zero of finitely generated free algebras in $\V$. However, this condition is necessary but not sufficient for an algebra $\A$ in the regularization $R(\mathcal{V})$ of $\mathcal{V}$ to be free, as the following example displays.

\begin{example}\label{ex: countcliff}

Let $I=\{0,1,2,3\}$ be the four-element lattice whose order $\le$ is $0 < 1,2 < 3$ ($1, 2$ incomparable).
Clearly $(I,\le)$ is the free (join) semilattice with zero on the set of generators $\{1,2\}$. Let $\G_i$ be the free group on the set of generator $\{x_{i}\}$ for every $i\in \{1,2,3\}$, while $\G_0$ is the trivial group (i.e. the free group on the empty set of generators). Let $\mathbf{C}$ be the P\l onka sum of the family $(\G_i)_{i\in I}$ over the semilattice $(I,\le)$ with homomorphisms $p_{13}$ and $p_{23}$ defined as the unique homomorphisms (extending the map) such that $p_{13}(x_{1})=x_{3}$ and $p_{23}(x_{2})=x_{3}$ (the remaining homomorphisms are defined in the obvious way). 

    
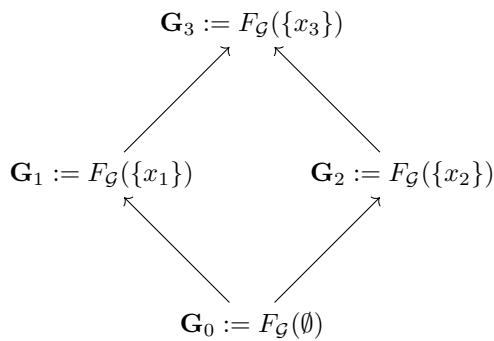
\begin{center}
        \begin{tikzpicture}[node distance=2cm, auto]
    \node (A3) at (0, 2)   {$\G_3 := F_{\mathcal{G}}(\{x_{3}\})$};
    \node (A1) at (-2, 0)  {$\G_1 := F_{\mathcal{G}}(\{x_{1}\})$};
    \node (A2) at (2, 0)   {$\G_2 := F_{\mathcal{G}}(\{x_{2}\})$};
    \node (A0) at (0, -2)  {$\G_0 := F_{\mathcal{G}}(\emptyset)$};

    \draw[->] (A0) -- (A1);
    \draw[->] (A0) -- (A2);
    \draw[->] (A1) -- (A3);
    \draw[->] (A2) -- (A3);
\end{tikzpicture}
    \end{center}

Therefore $\mathbf{C}$ is a P\l onka sum of free groups over the free (join) semilattice on two generators. But $\mathbf{C}$ is not a free algebra in $R(\mathcal{G})$, that is $\mathbf{C}$ is not a free Clifford semigroup. Indeed, suppose on the contrary that this is the case, then by \eqref{eq: generators} $\mathbf{C}$ would be free on the set of generators $\{x_{1}, x_{2}\}$. Upon indicating with $\cdot$ the multiplication in $\mathbf{C}$, we have $x_{1}\cdot^{\mathbf{C}} x_{2}=p_{13}(x_{1})\cdot^{\G_{3}} p_{23}(x_{2})=x_3\cdot^{\G_{3}} x_3=p_{23}(x_{2})\cdot ^{\G_{3}}p_{13}(x_{1})=x_{2}\cdot^{\mathbf{C}}x_{1}$, and since $x_{1}$, $x_{2}$ are generators of $\mathbf{C}$, we would have that the variety of Clifford semigroups would satisfy the identity $x_{1} \cdot x_{2} \approx x_{2}\cdot x_{1}$, in contradiction with the fact that not every Clifford semigroup is commutative.    
\end{example}


\begin{theorem}\label{th: caratterizzazione libere}
Let $\mathbf{A}$ be an algebra in $R(\mathcal{V})$, with $\mathcal{V}$ a strongly irregular variety. Then
$\mathbf{A}\in R(\V)$ is free on the set of generators $\{a_j\}_{j\in J}$ iff $\mathbf{A}=\PLA$ 
with $\mathbb{A} = (\{\mathbf{A}_i\}_{i\in I}, (I,\le), \{p_{ij}\}_{i\le j})$ such that:

\begin{enumerate}
    \item $(I,\le)$ is the free semilattice with zero on the set of generators $J$;
    
\item    for every $ i\in I, \mathbf{A}_i$ is a finitely generated $\V$-free algebra. 
In particular $\A_i$ has exactly $n$ generators $G^i = \{g^{i}_1, \dots, g^{i}_n\}$ iff $i = i_1\vee\dots\vee i_n$,  for $i_1,\dots,i_n\in I$ ($i_{k}\neq i_{m}$ for $k\neq m$)
, and if $i\in J$ then $\A_i$ the free algebra generated by the element $g^{i}_{1} = a_{i_1} = a_i$;

    \item for every $ i\in I$, $p_{{i_0}i}$ is the unique homomorphism from $\mathbf{A}_{i_0}$ into $\mathbf{A}_i$, where $i_0$ is the least element of $I$, while for every $ i_1,...,i_n \in J$ ($\forall n\in \mathbb{N}^+$) if $i=i_1\vee...\vee...\vee i_n$ then for every $ j\in \{1,...,n\}: p_{i_ji}:\mathbf{A}_{i_j} \rightarrow \mathbf{A}_{i}$ is the unique injective homomorphism extending the (injective) map $p^{0}_{i_ji}\colon \{a_{i_j}\}\to \A_i$, $a_{i_j}\mapsto p^{0}_{i_ji}(g^{i}_{j}) := g^{k}_{j}$. In particular, for $i_{0}\neq i\leq k$, $p_{ik}\colon \mathbf{A}_i\to\mathbf{A}_k$ is the unique injective homomorphism extending the (injective) map $p^{0}_{ik}\colon G^{i}\to\mathbf{A}_k$, $g^{i}_{j}\mapsto p^{0}_{ik}(g^{i}_{j}) := g^{k}_{j}$.
\end{enumerate}
\end{theorem}



\begin{proof} 
The proof runs essentially as that of \cite{Romanowskafree}. We include it here as it clarifies the structure of free algebras.  

\noindent
$(\Rightarrow)$ 
Let $\mathbf{A}$ be a free algebra in $R(\V)$ with generators $\{a_{j}\}_{j\in J}$.

\noindent
By Lemmas \ref{lemma: Lemma 1 libere} and \ref{lemma: Lemma 2 libere}, $\mathbf{A}$ is a P\l onka sum over the semilattice direct with zero $\mathbb{A} = (\{\mathbf{A}_i\}_{i\in I}, (I,\le), \{p_{ij}\}_{i\le j})$ such that:
\begin{itemize}
    \item $(I,\leq ) = (\mathcal{P}_{fin}(J),\subseteq)$;
    \item for every $ i\in I\setminus\{i_0\}, \A_i$ is a free finitely generated algebra in $\V$, while $\mathbf{A}_{i_0}=Sg^{\mathbf{A}}(\emptyset)$.
\end{itemize}
\noindent
We show that (1), (2) and (3) actually hold.

\noindent (1) $\mathcal{P}_{fin}(J)$ is the free semilattice with zero over the set of generators $X =\{\{j\} \;:\; j\in J\}$. Indeed, consider the map $f_0\colon X\to \mathbf{S}$, from the set $X$ to a semilattice with zero $\mathbf{S}$. Then the map $f\colon\mathcal{P}_{fin}(J) \rightarrow \mathbf{S}$ defined by: 

$$
f(\{j_1,..,j_n\}) = \begin{cases}
f_{0}(j_1)\vee\dots\vee f_{0}(j_n), \text{ if } \{j_1,..,j_n\}\in \mathcal{P}_{fin}(J)\setminus\emptyset, \\
   0_{\mathbf{S}}, \text{ if } \{j_1,..,j_n\} = \emptyset\\
\end{cases}$$

\noindent 
coincides with $f_0$ when restricted to a singleton $\{j\}$ . Moreover $f$ is a semilattice homomorphism. Indeed 
\begin{align*}
 f(\{j_{1},\dots, j_{k}\} \cup \{j'_{1},\dots, j'_{l}\}) = & f(\{j_1, \dots, j_k,j'_1,\dots, j'_l\})
 \\ =& f_{0}(\{j_1\})\vee\dots f_{0}(\{j_k\})\vee f_{0}(\{j'_1\})\vee\dots \vee f_{0}(\{j'_l\})
 \\ =& f(\{j_1\})\vee\dots f(\{j_k\})\vee f(\{j'_1\})\vee\dots \vee f(\{j'_l\}).   
\end{align*}
     
\noindent And this shows that $(I,\leq ) = (\mathcal{P}_{fin}(J),\subseteq)$ is a free semilattice with zero.
\\

\noindent
(2) By Lemma \ref{lemma: Lemma 2 libere}, $\A_{i_0} = \A_{\emptyset}$ and, by definition, $\A_{\emptyset} = Sg^{\A}(\emptyset)$. By Lemmas \ref{lemma: Lemma 1 libere} and \ref{lemma: Lemma 2 libere}, for every $i\in I\setminus \{i_0\}$, $\A_i$ is the free algebra in $\V$ with generators $G^{i} = \{g^{i}_{1},\dots,g^{i}_{n}\}$. Observe that, in case $i\in J$, i.e. it is a generator of the free semilattice (with zero) $(I,\leq)$, then, by construction and the fact that $I = \mathcal{P}_{fin}(J)$, $i$ corresponds to a singleton $J_0 = \{k_1\}\subseteq J$ and $G^{i} = \{g_{1}\}$ with $g_1 = a_{k_1} = a_1$: thus $\A_{i}$ is one generated by this single element. On the other hand, if $i\in I\setminus J$ then, by construction, the algebra $\A_i$ is generated by $G^{i} = \{g^{i}_1,\dots, g^{i}_{n}$\}. \\

\noindent (3) Firstly, observe that, if $i\in J$ (i.e. $i$ is chosen among the generators of the semilattice with zero $(I,\leq)$) then, by (2), $\A_i$ is $\V$-free generated by $a_i$, therefore, for every $k\in I$ such that $i\leq k$, the (injective) map $p^{0}_{ik}\colon \{a_i\}\to \A_k$ extends to an injective homomorphism $p_{ik}\colon\A_i\to \A_k$ (see e.g. \cite[Exercise 5, Ch.10]{BuSa00}). Similarly, if $i\in I\setminus J\cup\{i_0\}$ then $i = i_1\vee\dots\vee i_n$, with $i_1,\dots, i_n\in J$ ($(I,\leq)$ is free and generated by $J$) and, by (2) $\A_i$ is a free algebra in $\V$ generated by $G^{i} = \{g_1,\dots,g_n\}$. Moreover, for every $k\in I$ such that $i\leq k$, $\A_k$ is $\V$-free and generated by $G^{k} = \{g_1, \dots,g_m\}$ with $n\leq m$ (as $(I,\leq) = (\mathcal{P}_{fin}^{\ast}(J)),\subseteq$ by (1)), therefore it is naturally defined an injective map $p^{0}_{ik}\colon G^{i}\to\A_{k}$, $g^{i}_{j}\mapsto g^{k}_{j}$ (for every $j\in\{1,\dots, n\}$) and since $\A_i$ is $\V$-free, $p^{0}_{ik}$ extends to an injective homomorphism $p_{ik}\colon\A_i\to\A_k$. \\

\noindent
$(\Leftarrow)$
By hypothesis, $\A\in R(\V)$ is the P\l onka sum over a semilattice direct systems satisfying conditions (1)-(2)-(3). In particular, the partition function corresponding to the P\l onka sum decomposition of $\A$ is given by the interpretation of the term function $x\cdot y$ witnessing the strong irregularity of $\V$.
\noindent
First, let's prove that $\mathbf{A}=Sg^\mathbf{A}(\{a_j\})_{j\in J}$. By (1), $(I,\leq)$ is the free semilattice (with zero) generated by $J$, so let $i\in I$, then $i = i_1\vee\dots\vee i_{n}$, for some $i_{1},\dots, i_n\in J$. Let $a_{i_k}\in A_{i_k}$, for some $k\in \{1,\dots,n\}$. Recall from Theorem \ref{th: Plonka dec. theorem} that $p_{i_{k}i}(a_{i_k}) = a_{i_k}\cdot^{\mathbf{A}} x$, with $x$ an arbitrary element chosen in $A_i$. By (2), $\A_i$ is $\V$-free on the set of free generators $G^{i} = \{g^{i}_{1},\dots,g^{i}_{n}\}$. Assume that $a_{i_1}\in A_{i_1}$, ..., $a_{i_n}\in A_{i_n}$. Then $a_{i_1} \cdot^{\A} ... \cdot^{\A} a_{i_n} = p_{i_{1}i}(a_{i_1})\cdot^{\A_i}...\cdot^{\A_i} p_{i_{n}}i(a_{i_n}) = g^{i}_{1}\cdot^{\A_i} ... \cdot^{\A_i} g^{i}_{n}$, where 
last equality holds by (3).
For a generic generator $g^{i}_{k}$ of $\A_i$, we have that $g_j^i=p_{i_{j}i}(a_{i_k})= a_{i_k}\cdot^{\A} (a_{i_1}\cdot^{\A} ... \cdot^{\A} a_{i_n}) = a_{i_k}\cdot^{\A} a_{i_1}\cdot^{\A} ... \cdot^{\A} a_{i_n}$, where the last equality follows using the fact that $\cdot^{\A}$ is a partition function (i.e. in particular $(A,\cdot)$ is a left normal band).
We have hence shown that any generator of $\A_i$ can be written as a product of elements in $\{a_j\}_{j\in J}$, therefore $\A$ is generated by $\{a_j\}_{j\in J}$.

\noindent
We now prove that $\A$ has the universal mapping property for $R(\V)$ over $\{a_j\}_{j\in J}$. Let $\B\in R(\V)$ and $h_0\colon \{a_j\}_{j\in J} \rightarrow \mathbf{B}$ be a map. Since $\B \in R(\V)$, by Theorem \ref{th: regularization of st.irregular varieties}, $\B$ is P\l onka sum over a semilattice direct system $\mathbb{B}=(\{B_l\}_{l\in L},(L,\le^L),\{q_{uv}\}_{u\le v})$. The map $h_0$ induces a map $f_0\colon J \rightarrow L$ defined as follows: for $j\in J$, $f_0(j) = l$, where $l$ is the index of the fiber $\B_l$ such that $h_0(a_j)\in B_l$ (the fibers of a P\l onka sum are disjoint, so $f_0$ is well-defined). 

\noindent
Let $f\colon I\rightarrow L$ be the semilattice homomorphism extending $f_0$ (by (1) $(I,\le)$ is free on the set of generators $J$). 
Let $i\in I$, hence $i = i_1\vee\dots\vee i_{n}$, for some $i_1,\dots,i_n\in J$ (as $J$ generates $i$ by (1)). Moreover, by (2), the algebras $\A_{i_1},\dots,\A_{i_n}$ are one generated by the elements $a_{i_1},\dots, a_{i_n}$, respectively. Therefore we can consider the elements $h_0(a_{i_1})\in B_{f(i_1)},\dots, h_0(a_{i_n})\in B_{f(i_n)}$. Observe that since $f$ is a semilattice homomorphism we have $f(i_1)\vee\dots\vee f(i_n) = f(i_1\vee\dots \vee i_n) = f(i)$. Therefore we can define a map $h_0^{i}\colon G^i\to \B_{f(i)}$ as $g^{i}_{k}\mapsto h_{0}^{i}(g_{k}^{i})=q_{f(i_k)f(i)}(h_0(a_{i_k}))$. Let $h^{i}\colon \A_i\to \B_{f(i)}$ the homomorphism extending $h_{0}^{i}$, and for $i=i_0$ let $h^{i_0}$ be the unique homomorphism $\A_{i_0} \rightarrow \B_{l_0}$.
Let us define a map $h\colon \A\to\B$ as $a\mapsto h(a):= h^{i}(a)$, for $a\in A_i$.
In order to verify that $h$ is a homomorphism we apply the  categorical equivalence established in \cite[Th. 3.3.8]{Bonziobook} (see also \cite{SBLogicaUniv}, \cite{BonzioLoiDuality}), so we show that $(f,(h^i)_{i\in I})$ is a morphism of semilattice direct systems with zero.

\noindent
As for every $i\in I$ we have that $h^i \in Hom(\A_i, \B_{f(i)})$, the following diagram is commutative for every $i\in I$, since for every $\tau$-algebra $\mathbf{C}$ there exists one and only one homomorphism $\A_{i_0} \rightarrow \mathbf{C}$.

\begin{center}

\begin{tikzpicture}
\draw (-4,0) node {$ \B_{l_0} $};
	\draw (4,0) node {$\B_{f(i)}$};
	
	\draw [line width=0.8pt, ->] (-3.6,0) -- (3.6,0);
	\draw (0,-0.4) node {\begin{footnotesize}$q_{l_0f(i)}$\end{footnotesize}};
	
	\draw [line width=0.8pt, <-] (3.3,3) -- (-3.3,3);
	\draw (0, 3.3) node {\begin{footnotesize}$p_{i_0i}$\end{footnotesize}};

	\draw (-4,3) node {$\A_{i_0}$}; 
	
	\draw (4,3) node {$\A_{i}$};
		
	\draw [line width=0.8pt, ->] (-4,2.6) -- (-4,0.4);
	\draw (-4.6,1.7) node {\begin{footnotesize}$h^{i_0}$\end{footnotesize}};
	\draw (4.6,1.7) node {\begin{footnotesize}$h^{i}$\end{footnotesize}};
	\draw [line width=0.8pt, <-] (3.9,0.4) -- (3.9,2.6);
	
\end{tikzpicture}
\end{center}

\noindent
So to conclude the proof is sufficient to show that for every $i_0\neq i\le i'$ the following diagram is commutative 

\begin{center}

\begin{tikzpicture}
\draw (-4,0) node {$ \B_{f(i)} $};
	\draw (4,0) node {$ \B_{f(i')} $};
	
	\draw [line width=0.8pt, ->] (-3.6,0) -- (3.6,0);
	\draw (0,-0.4) node {\begin{footnotesize}$q_{f(i)f(i')}$\end{footnotesize}};
	
	\draw [line width=0.8pt, <-] (3.3,3) -- (-3.3,3);
	\draw (0, 3.3) node {\begin{footnotesize}$p_{ii'}$\end{footnotesize}};

	\draw (-4,3) node {$\A_i$}; 
	
	\draw (4,3) node {$\A_{i'}$};
		
	\draw [line width=0.8pt, ->] (-4,2.6) -- (-4,0.4);
	\draw (-4.6,1.7) node {\begin{footnotesize}$h^i$\end{footnotesize}};
	\draw (4.6,1.7) node {\begin{footnotesize}$h^{i'}$\end{footnotesize}};
	\draw [line width=0.8pt, <-] (3.9,0.4) -- (3.9,2.6);
	
\end{tikzpicture}
\end{center}

\noindent Since every algebra $\A_i$ is a $\V$-free algebra, it is enough to check the commutativity of the above diagram on the generators.

\begin{center}

\begin{tikzpicture}
\draw (-4,0) node {$ \B_{f(i)} $};
	\draw (4,0) node {$ \B_{f(i')} $};
	
	\draw [line width=0.8pt, ->] (-3.6,0) -- (3.6,0);
	\draw (0,-0.4) node {\begin{footnotesize}$q_{f(i)f(i')}$\end{footnotesize}};
	
	\draw [line width=0.8pt, <-] (3.3,3) -- (-3.3,3);
	\draw (0, 3.3) node {\begin{footnotesize}$p_{ii'}$\end{footnotesize}};

	\draw (-4,3) node {$G_i$}; 
	
	\draw (4,3) node {$G_{i'}$};
		
	\draw [line width=0.8pt, ->] (-4,2.6) -- (-4,0.4);
	\draw (-4.6,1.7) node {\begin{footnotesize}$h^i_0$\end{footnotesize}};
	\draw (4.6,1.7) node {\begin{footnotesize}$h^{i'}_0$\end{footnotesize}};
	\draw [line width=0.8pt, <-] (3.9,0.4) -- (3.9,2.6);
	
\end{tikzpicture}
\end{center}

\noindent 
Assume that $i\neq i_0$ and let $g^{i}_{k}\in G_i$. Then  $q_{f(i)f(i')}(h^{i}_{0}(g^{i}we_{k})) =vq_{f(i)f(i')}(q_{f(i_{k})f(i)}(h_{0}(a_{i_k})))= q_{f(i_{k})f(i')}(h_{0}(a_{i_k}))$,

\noindent 
where the first equality holds since by (3) the generator $a_{i_k}$ is sent into the generator $g^{i}_{k}$. \\
\noindent
On the other hand, $h^{i'}_{0}(p_{ii'}(g_{k}^{i})) = h^{i'}_{0}(g_{k}^{i'}) = q_{f(i_{k})f(i')}(h_{0}(a_{i_k}))$, where the first equality holds once again by hypothesis (3). The last equality holds in virtue of the injectivity of the map $p_{ii'}$ (again by (3)), from which it follows that the element $a_{i_k}$ is the unique element sent into $g^{i}_{k}$ and $g^{i'}_{k}$ by $p{i_{k}i}$ and $p_{i_{k}i'}$, respectively.
\end{proof} 

\begin{remark}
 Observe that the Clifford semigroup $\mathbf{C}$ introduced in Example \ref{ex: countcliff} is \emph{not free} as the two generators $x_{1}$ and $x_{2}$ are mapped (by transition homomorphisms) into the same generator $x_{3}$, violating condition $(3)$ in Theorem \ref{th: caratterizzazione libere}. Moreover,
 also condition (2) in Theorem \ref{th: caratterizzazione libere} is violated as would impose the free group $\G_{3}$ to be two-generated.
\end{remark}


\section{Congruences}



Throughout the whole section we will work with a semilattice direct system $\mathbb{A}=(\{\A_i\}_{i\in I}, (I, \le), \{p_{ij}\}_{i\le j})$ of algebras belonging to a strongly irregular variety $\V$ ($\A_{i}\in\V$ for every $i\in I$), whose strong irregularity is witnessed by a binary term $x\cdot y$. 

We iniziate our investigation by understanding the essential structural features of any congruence on $\PLA$. 

\begin{definition}\label{def: cong1}
    Let $\theta \in Con(\PLA)$. Define:

\begin{enumerate}
    \item For every $ (i,j) \in I $, let $ \theta_{ij}:= \theta \cap (A_i \times A_j)$;
    \item $S_{\theta}:=\{(i,j) \in I \times I \mid \theta_{ij} \neq \emptyset\}$, i.e. $(i,j) \in I$ iff there exist elements $a\in A_{i}$, $b\in A_{j}$ such that $(a,b)\in\theta$;
    \item $C_\theta=\{(i,j) \in I \times I \mid \{(i, i \vee j), (j, i \vee j)\} \subseteq S_{\theta}\}$.
\end{enumerate}
\end{definition}

We will refer to $\theta_{ij}$ as the \emph{fibers} of the congruence $\theta$ (the use of the term will become clear later). The strong irregularity of $\mathcal{V}$ provides us with the following fundamental lemma, which sheds light on the algebraic structure of $S_\theta$.

\begin{lemma}\label{lemma: congruence1}
    Let $\theta \in Con(\PLA)$, then for every $(i,j) \in S_\theta $ and  for any $ a \in A_i $, $ (a,p_{ii\vee j}(a)) \in \theta$. Moreover, $S_{\theta}$ is a reflexive and symmetric subsemilattice of $\mathbf{I \times I}$ satisfying the following condition: 
$$
 \text{for every } i,j,k\in I \text{ such that } (i,j),(j,k) \in S_{\theta},  (i, i \vee k) \in S_{\theta}. \hspace{2cm} \text{\emph{(Upper Transitivity)}}
 $$
\end{lemma}

\begin{proof}
    Let $(i,j)\in S_{\theta}$, then there exist $(c,d)\in \theta_{ij}$, i.e. $(c,d)\in\theta$ with $c\in A_{i}$ and $d\in A_{j}$. Since $\theta$ is a congruence on $\PLA$ we have $(a, p_{ii\vee j}(a))=(a\cdot^{\PLA} c, a \cdot^{\PLA} d) \in \theta$, for any $a\in A_{i}$, where the equality follows from $a\cdot^{\PLA} c = a$ (as $a,c\in A_{i}$) and $p_{i,i\vee j}(a) = a\cdot^{\PLA} x$, for any $x\in A_j$. Reflexivity and symmetry of $S_{\theta}$ follow immediately from the definition of $S_{\theta}$. To see that $S_{\theta}$ is a subsemilattice of $\mathbf{I \times I}$, let $(i,j),(u,v) \in S_{\theta}$, therefore there exist $(a,b)\in \theta_{ij}, (c,d)\in \theta_{uv}$, so $(a\cdot^{\PLA} c, b\cdot^{\PLA}d)\in \theta_{i \vee u, j \vee v}$, consequently $(i\vee u, j \vee v) \in S_{\theta}$. \\
\noindent
To see that $S_{\theta}$ is also upper transitive, consider $(i,j),(j,k) \in S_{\theta}$ and $a\in A_i$. Since $(i,j) \in S_{\theta}$, we have $(a,p_{ii\vee j}(a))\in \theta$. Furthermore, $(i,j),(j,k)\in S_{\theta} $ implies $ (i \vee j, j \vee k) \in S_{\theta}$ (by preservation of join), consequently (thanks to what proved above) $(p_{ii\vee j}(a), p_{ii\vee j \vee k}(a))=(p_{ii\vee j}(a),p_{i\vee j, (i\vee j) \vee (j \vee k)}(p_{ii\vee j}(a))) \in \theta$. So by transitivity of $\theta$ we have $(a,p_{ii\vee j \vee k}(a)) \in \theta$. Moreover, $(i,j),(j,k) \in S_{\theta}$, then $(i\vee k, j)=(i\vee k, j \vee j) \in S_{\theta}$, so we have $(p_{ii\vee k}(a), p_{ii\vee j \vee k}(a))=(p_{ii\vee k}(a), p_{i\vee k,i\vee j \vee k}((p_{ii\vee k}(a)))) \in \theta$. Consequently, by transitivity of $\theta$, $(p_{ii\vee j}(a),p_{ii\vee k}(a)) $, whence $(a,p_{ii\vee k}(a)) \in \theta$, that is $(i,i \vee k) \in S_{\theta}$.
\end{proof}

 The weaker form of transitivity, dubbed \emph{upper transitivity}, introduced in the previous lemma unfortunately cannot be replaced in general with transitivity, as illustrated by the following example. 

\begin{example}
    Consider the P\l onka sum (defined over a semilattice direct system) of (three) lattices given in the following picture, where the two bottom fibers are trivial lattices and the top fiber is a two-element lattice.

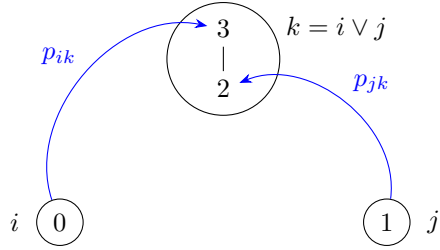
\begin{center}
\begin{tikzpicture}[node distance=2cm, auto, >=Stealth]
  \node (circ23) [circle, draw, minimum size=1.5cm] at (0,0) {};
  \node (2) at (0, -0.4) {$2$};
  \node (3) at (0, 0.4) {$3$};

  \node (0_inf) [circle, draw, below left=of circ23] {$0$};
  \node (1_inf) [circle, draw, below right=of circ23] {$1$};

  \node [left=0.1cm of 0_inf] {$i$};
  \node [right=0.1cm of 1_inf] {$j$};
  \node [right=0.5cm of 3] {$k=i \vee j$};

  \draw[-] (3) -- (2);

  \draw[->, bend left=60, color=blue] (0_inf) to node {$p_{ik}$} (3);
  \draw[->, bend right=60, color=blue] (1_inf) to node [swap] {$p_{jk}$} (2);

\end{tikzpicture}

\end{center}

\noindent 
It is immediate to check that the congruence $\theta$ generated by $\{(0,3), (1,2)\}$ results into $S_\theta$ not being transitive (since $(i,k), (k,j) \in S_\theta$, but $(i,j) \notin S_\theta$). 
\end{example}

 The following lemma provides a characterization for the transitivity of $S_{\theta}$.

\begin{lemma}\label{lemma: congruence3}
    Let $\theta \in Con(\PLA)$ and $i,j,k \in I$ such that $(i,j),(j,k) \in S_{\theta}$.
    Then the following are equivalent: 
    \begin{enumerate}
        \item $(i,k) \in S_{\theta} $; 
        \item $(p_{ii\vee k} \times p_{ki\vee k})^{-1}(\theta_{i\vee k, i \vee k}) \neq \emptyset$.
    \end{enumerate}  
\end{lemma}
\begin{proof}
 (1) $\Rightarrow$ (2). By hypothesis we have $(i,k) \in S_\theta$, so there exists $(a,c) \in \theta_{ik}$. By Lemma \ref{lemma: congruence1} (in particular, upper transitivity) we have $(i,i \vee k)$ and $(k, i \vee k) \in S_{\theta}$ (observe that $S_\theta$ is symmetric), therefore by Lemma \ref{lemma: congruence1} we also have $(a,p_{ii\vee k}(a)), (c, p_{ki\vee k}(c)) \in \theta$ and, using the transitivity of $\theta $, $(p_{ii\vee k}(a), p_{ki\vee k}(c)) \in \theta$, that is $(p_{ii\vee k} \times p_{ki\vee k})^{-1}(\theta_{i\vee k, i \vee k}) \neq \emptyset$.\\
\noindent        
(2) $\Rightarrow$ (1). Let $(a,c) \in A_{i} \times A_k$ such that $(p_{ii\vee k}(a), p_{ki\vee k}(c)) \in \theta$. By Lemma \ref{lemma: congruence1} we have $(a,p_{ii\vee k}(a)), (c, p_{ki\vee k}(c)) \in \theta $. So $(a,c) \in \theta$ by symmetry and transitivity of $\theta$, therefore $(i,k) \in S_{\theta}$.
   \end{proof}

 In the subsequent lemma, we study the algebraic structure of $C_\theta$ and its close relationship with $S_\theta$.

\begin{lemma}\label{lemma: congruence6}
    Let $\theta \in Con(\PLA)$, then the following hold: 
    \begin{itemize}
        \item[(i)] $C_\theta \in Con(\mathbf{I})$;
        \item[(ii)] for every $(i,j) \in I \times I$, $ (i,j) \in S_{\theta} $ if and only if $ (i,j) \in C_{\theta} $ and $
    (p_{ii\vee j} \times p_{ji\vee j})^{-1}(\theta_{i\vee j, i \vee j}) \neq \emptyset. $    
\end{itemize}

\noindent In particular, $C_\theta=Cg^{\mathbf{I}}(S_\theta)$.
\end{lemma}

\begin{proof}
(i). Reflexivity and symmetry are immediate from the definition of $C_\theta$ and Lemma \ref{lemma: congruence1}. It remains to prove transitivity and join preservation.\\ 
\noindent Let $(i,j)$, $(j,k) \in I \times I$ be such that $(i,j),(j,k) \in C_{\theta}$, which means that $\{(i, i \vee j), (j, i \vee j)\}, \{(j, j \vee k),(k, j \vee k)\}\subseteq S_\theta$. Then, since (by Lemma \ref{lemma: congruence1}) $S_{\theta} \le \mathbf{I \times I}$, $(i \vee j,\ i \vee j \vee k) \in S_\theta$, so by Lemma \ref{lemma: congruence1} (upper transitivity of $S_\theta$) we deduce $(i, i \vee j \vee k) \in S_\theta$. Analogously, we also get $(k, i \vee j \vee k) \in S_\theta$, and so $\{(i, i \vee k), (k, i \vee k)\} \subseteq S_\theta$ (by upper transitivity) that is $(i,k) \in C_{\theta}$, showing transitivity.\\
\noindent Let $(i_1, j_1), (i_2, j_2) \in C_{\theta}$, that is $\{(i_1, i_1 \vee j_1), (j_1, i_1 \vee j_1), (i_2, i_2 \vee j_2), (j_2, i_2 \vee j_2)\} \subseteq S_{\theta}$, so, since $S_\theta \le \mathbf{I \times I}$, $\{(i_1 \vee i_2, (i_1 \vee i_2)\vee(j_1 \vee j_2)), (j_1 \vee j_2, (i_1 \vee i_2) \vee (j_1 \vee j_2))\} \subseteq S_\theta$, i.e. $(i_1 \vee i_2, j_1 \vee j_2) \in C_{\theta}$.\\
\noindent        
(ii). $(\Rightarrow)$ Clearly $S_\theta \subseteq C_\theta$ (since $S_\theta \le \mathbf{I \times I}$). Moreover, suppose $(i,j) \in S_\theta$, then we also have $(i,j),(j,j) \in S_\theta$ (by reflexivity) and applying Lemma \ref{lemma: congruence3} we deduce $(p_{ii\vee j} \times p_{ji \vee j})^{-1}(\theta_{i\vee j, i \vee j}) \neq \emptyset$. \\
\noindent
$(\Leftarrow)$ suppose $(i,j) \in C_{\theta}$ and $(p_{ii\vee j} \times p_{ji\vee j})^{-1}(\theta_{i\vee j,i \vee j}) \neq \emptyset$, then $\{(i, i \vee j),(j, i \vee j)\} \subseteq S_{\theta}$ and $(p_{ii\vee j} \times p_{ji\vee j})^{-1}(\theta_{i\vee j,i \vee j}) \neq \emptyset$, so by Lemma \ref{lemma: congruence3} we get $(i,j) \in S_{\theta}$.\\
    \noindent
 Finally, let $C$ be any congruence of $\I$ extending $S_{\theta}$ and $(i,j) \in C_{\theta}$, then by definition we have $\{(i,i \vee j), (j, i \vee j)\} \subseteq S_\theta \subseteq C \in Con(\mathbf{I})$, so by transitivity of $C$ we get $(i,j) \in C$. This shows that $C_\theta=Cg^{\mathbf{I}}(S_\theta)$.
\end{proof}

 It turns out that, in some particular, yet relevant, cases, $S_\theta=C_{\theta}$, that is $S_{\theta} \in Con(\mathbf{I})$.

\begin{lemma} \label{corollary: lemma5}
Let $\theta \in Con(\PLA)$ with the algebraic language of $\PLA$ containing constants (or $\V$ containing an algebraic constant) or $\I$ be a chain. Then $S_\theta \in Con(\mathbf{I})$.

    

    
\end{lemma}

\begin{proof} 
In view of Lemma \ref{lemma: congruence1},
it is sufficient to prove that $S_\theta$ is transitive. Suppose $\I$ is a chain. Assume w.l.o.g. $i \le j \le k$. We have two cases: $(i,j),(j,k) \in S_{\theta}$ and $(i,k), (k,j) \in S_\theta$. In the first case, by Lemma \ref{lemma: congruence1}, we have $(i,k)=(i,i\vee k) \in S_{\theta}$, while in the second one $(i,j)=(i,k \vee j) \in S_{\theta}$. \\
\noindent
Suppose that the algebraic language of $\PLA$ contains constants. We show, more in general, that, if for every $ i,j,k\in I: Im(p_{ii\vee j \vee k}) \cap Im(p_{ki\vee j \vee k}) \neq \emptyset$ then $S_{\theta}$ is transitive. Indeed, suppose $(i,j),(j,k) \in S_{\theta}$, then by Lemma $\ref{lemma: congruence1}$ we have $(i, i \vee k), (k, i \vee j) \in S_{\theta}$. Since $S_{\theta}$ preserves joins and is reflexive, from the previous we deduce $(i, i \vee j \vee k), (k, i \vee j \vee k) \in S_{\theta}$. Let $c \in Im(p_{ii\vee j \vee k}) \cap Im(p_{ki\vee j \vee k})$, then by definition there exist $a \in A_i, b \in A_k$ such that $p_{ii\vee j \vee k}(a)=c=p_{k i \vee j \vee k}(b)$. By Lemma \ref{lemma: congruence1} we have $(a,c)=(a, p_{ii\vee j \vee k}(a)) \in \theta$ and $(b,c)=(b,p_{ki\vee j \vee k}(a)) \in \theta$, so $(a,b) \in \theta$ by transitivity of $\theta$. Therefore $(i,k) \in S_{\theta}$. In case $\PLA$ contains constants then the above condition is clearly satisfied, thus $S_{\theta}$ is a congruence.  
        \end{proof}

In the following lemma, we analyze the structure of the fibers of $\theta$.

\begin{lemma}\label{lemma: congruence4}
    Let $\theta \in Con(\PLA)$, then: 
    \begin{itemize}
        \item[(i)]  for every $(i,j) \in I \times I $: $ \theta_{ii} \subseteq (p_{ii\vee j} \times p_{ii\vee j})^{-1}(\theta_{i\vee j,i\vee j})$;
        \item[(ii)] for every $i\in I$ and  $(i,j) \in S_\theta $: $ \theta_{ii} = (p_{ii\vee j} \times p_{ii\vee j})^{-1}(\theta_{i\vee j,i\vee j})$;

        \item[(iii)] for every $(i,j) \in S_\theta: \theta_{ij}=(p_{ii\vee j} \times p_{ji\vee j})^{-1}(\theta_{i\vee j,i\vee j})$.
    \end{itemize}

\end{lemma}

\begin{proof}
(i) Let $(a,b)\in \theta_{ii}$, i.e. $(a,b)\in\theta$ with $a,b\in A_{i}$. Then, for any $c\in A_j$, we have $(p_{ii\vee j}(a),p_{ii\vee j}(b))=(a\cdot^{\PLA} c,b\cdot^{\PLA} c)\in\theta$, showing the desired inclusion.\\
    \noindent 
(ii) Let $(i,j) \in S_{\theta}$ and suppose $(a,b) \in (p_{ii\vee j} \times p_{ii\vee j})^{-1}(\theta_{i\vee j, i \vee j})$, that is $(p_{ii\vee j}(a),p_{ii\vee j}(b)) \in \theta_{i\vee j, i \vee j}$. By Lemma \ref{lemma: congruence1} we also have $(a, p_{ii\vee j}(a)) \in \theta$ and $(b, p_{ii\vee j}(b)) \in \theta$, therefore by transitivity of $\theta$ we deduce $(a,b) \in \theta_{ii}$. So for every $(i,j) \in S_{\theta}: (p_{ii\vee j} \times p_{ji\vee j})^{-1}(\theta_{i\vee j, i \vee j}) \subseteq \theta_{ii}$.\\
\noindent
(iii) Suppose $(i,j)\in S_{\theta}$. Let $(a,b) \in \theta_{ij}$. By Lemma \ref{lemma: congruence1} we also have $(a, p_{ii\vee j}(a)), (b, p_{ji\vee j}(b)) \in \theta$, so $(p_{ii\vee j}(a), p_{ji\vee j}(b)) \in \theta$, that is $(a,b) \in (p_{ii\vee j} \times p_{ji\vee j})^{-1}(\theta_{i\vee j, i \vee j})$. Conversely, let $(a,b) \in (p_{ii\vee j} \times p_{ji\vee j})^{-1}(\theta_{i\vee j, i \vee j})$, then $(p_{ii\vee j}(a), p_{ji\vee j}(b)) \in \theta_{i\vee j, i \vee j}$, but also $(a,p_{ii\vee j}(a)), (b, p_{ji\vee j}(b)) \in \theta$ by Lemma \ref{lemma: congruence1}, so $(a,b) \in \theta_{ij}$ by symmetry and transitivity of $\theta$.
\end{proof} 

For every $i\in I$, we will call the sets $\theta_{ii}$ \emph{pure fibers} and $\theta_{ij}$ with $(i,j) \in S_{\theta} \setminus\Delta_{\mathbf{I}}$ \emph{mixed fibers}. Note that the use of the term fibers is in accordance with the fact that, for $\theta \le \PLA \times \PLA \in R(\mathcal{V})$,  such subsets effectively constitutes the fibers (in $\mathcal{V}$) of the P\l onka decomposition of $\theta$ as an algebra in $R(\mathcal{V})$.\\

The following technical lemma is a sort of converse of the results proved in so far, namely it shows how to construct the objects encountered so far, starting from a family of congruences on the fibers of a P\l onka sum.
 

\begin{lemma}\label{lemma: congruence5}
Let $\PLA\in R(\V)$, 
$\{\theta_{ii}\}_{i\in I}$ be a family of congruences of the fibers ($\theta_{ii} \in Con(\A_i)$ for every $i\in I$) and for every $ i,j\in I $, $ \theta_{ii} \subseteq (p_{ii\vee j} \times p_{ii\vee j})^{-1}(\theta_{i \vee j, i \vee j})$. Then the following hold:

    \begin{itemize}
        \item[(i)] if $S$ is a reflexive, symmetric and upper transitive subsemilattice of $\mathbf{I \times I}$, then 
        $C_S=\{(i,j) \in I \times I \mid \{(i, i \vee j), (j, i \vee j)\} \subseteq S\} \in Con(\mathbf{I}).$
        \item[(ii)] If $C \in Con(\mathbf{I})$, then $S_{C} = \{(i,j)\in C \;|\; (p_{ii\vee j} \times p_{ji\vee j})^{-1}(\theta_{i\vee j, i \vee j}) \neq \emptyset \}$
is a reflexive, symmetric and upper transitive subsemilattice of $\mathbf{I \times I}$.
        \item[(iii)] $C_{S_C}=C$ for every $C \in Con(\I)$, and for every $ S \le \mathbf{I \times I}$ reflexive, symmetric and upper transitive, if $ (i,j) \in S $ imply $ (p_{ii\vee j} \times p_{ji\vee j})^{-1}(\theta_{i\vee j,i\vee j}) \neq \emptyset$, then $S_{C_S}=S$.

    \end{itemize}
\end{lemma}

\begin{proof}
(i) The proof runs exactly as that of of Lemma \ref{lemma: congruence6}-(i). \\
\noindent
(ii) Observe that, for every $ (i,j), (h,k) \in I \times I $ such that $ (p_{i,i\vee j} \times p_{j,i\vee j})^{-1}(\theta_{i\vee j, i \vee j}) \neq \emptyset $ it holds that $ (p_{i \vee h, i \vee j \vee h \vee k} \times p_{j \vee k, i \vee j \vee h \vee k})^{-1}(\theta_{i\vee j \vee h \vee k, i \vee j \vee h \vee k}) \neq \emptyset.$ In fact, by hypothesis $\theta_{i\vee j, i \vee j} \subseteq (p_{i\vee j, i \vee j \vee h \vee k} \times p_{i\vee j, i \vee j \vee h \vee k})^{-1}(\theta_{i\vee j \vee h \vee k, i \vee j \vee h \vee k})$, so $\emptyset \neq (p_{ii\vee j} \times p_{ji\vee j})^{-1}(\theta_{i\vee j, i \vee j}) \subseteq (p_{ii\vee j \vee h \vee k} \times p_{ji\vee j \vee h \vee k})^{-1}(\theta_{i\vee j \vee h \vee k, i \vee j \vee h \vee k})=(p_{ii\vee h} \times p_{jj\vee k})^{-1}((p_{i\vee h,i\vee j \vee h \vee k} \times p_{j\vee k, i \vee j \vee h \vee k})^{-1}(\theta_{i\vee j \vee h \vee k, i \vee j \vee h \vee k}))$, therefore \\
\noindent 
$(p_{i\vee h,i\vee j \vee h \vee k} \times p_{j\vee k, i \vee j \vee h \vee k})^{-1}(\theta_{i\vee j \vee h \vee k, i \vee j \vee h \vee k}) \neq \emptyset$. It follows immediately that $S_C$ is a reflexive, symmetric and upper transitive subsemilattice of $\mathbf{I \times I}$.\\
\noindent
(iii) By definition we have that $(i,j) \in C_{S_C} $ if and only if $ \{(i, i \vee j), (j, i \vee j)\} \subseteq S_C $, that is $ (i,i \vee j), (j, i \vee j) \in C $ and $(p_{ii\vee j} \times p_{i\vee j,i\vee j})^{-1}(\theta_{i\vee j, i \vee j}) \neq \emptyset $, $(p_{ji\vee j} \times p_{i\vee j, i \vee j})^{-1}(\theta_{i\vee j, i \vee j}) \neq \emptyset$, which is equivalent (just) to $ (i,i \vee j),(j, i \vee j) \in C $ (indeed for $ a\in A_i$ and $ b \in A_j$ we have $(a, p_{ii\vee j}(a)) \in (p_{ii\vee j} \times p_{i\vee j,i\vee j})^{-1}(\theta_{i\vee j, i \vee j})$ and $(b, p_{ji\vee j}(b)) \in (p_{ji\vee j} \times p_{i\vee j,i\vee j})^{-1}(\theta_{i\vee j, i \vee j})$). Finally, the fact that $C$ is a congruence of $\I$ guarantees that $(i,j) \in C$.
\end{proof}

We are now ready to provide a characterization of the congruences of a P\l onka sum. In the statement we will adopt the nomenclature introduced in Lemma \ref{lemma: congruence5}. We first provide the following.

\begin{definition} \label{def: congsist}
    Let $\mathbb{A}=(\{\A_i\}_{i\in I}, (I, \le), \{p_{ij}\}_{i \le j})$ be a semilattice direct system. A \emph{congruence} on $\mathbb{A}$ is a pair $(C, \{\theta_{ii}\}_{i\in I})$ satisfying the following conditions:

        \begin{itemize}
        \item[(i)] $C \in Con(\mathbf{I})$;
        \item[(ii)] $ \theta_{ii} \in Con(\A_i)$ for every $i\in I$;
        \item[(iii)] $\theta_{ii} \subseteq (p_{ii\vee j} \times p_{ii\vee j})^{-1}(\theta_{i\vee j,i\vee j})$ for every $(i,j)\in I\times I$;
        \item[(iv)] $\theta_{ii} = (p_{ii\vee j} \times p_{ii\vee j})^{-1}(\theta_{i\vee j,i\vee j})$ for every $(i,j) \in S_C$.
    \end{itemize}

 We denote with $Con(\mathbb{A})$ the set of congruences on $\mathbb{A}$. It is immediate to check that $Con(\mathbb{A})$ forms a (algebraic) lattice when equipped with the partial order relation given by $(C, \{\theta_{ii}\}_{i\in I}) \subseteq (C', \{\theta_{ii}'\}_{i\in I})$ if and only if $C \subseteq C'$ and $\theta_{ii} \subseteq \theta_{ii}'$ for every $i\in I$. 
    
\end{definition}



\begin{theorem}\label{theorem: congruence5}
Let $\PLA\in R(\V)$. 
Then the lattices $Con(\A)$ and $Con(\mathbb{A})$ are isomorphic. More precisely: 
\begin{enumerate}
    \item if $\theta$ is a congruence on $\PLA$, then $(C_\theta, \{\theta_{ii}\}_{i\in I})$, where $C_\theta$ and $\theta_{ii}$ for every $i\in I$ are as in Definition \ref{def: cong1}, is a congruence on $\mathbb{A}$.
    \item if $(C, \{\theta_{ii}\}_{i\in I})$ is a congruence on $\mathbb{A}$, let $ \theta_{ij}:=(p_{ii\vee j} \times p_{ji\vee j})^{-1}(\theta_{i\vee j,i \vee j})$ for every $(i,j) \in S_C\setminus\Delta_{\mathbf{I}}$, then
    
    \begin{center}
        $\displaystyle \theta:= \bigcup_{(i,j) \in S_{C}}{\theta_{ij}}$
    \end{center} 
    
    \noindent is a congruence on $\PLA$.
    \item The two previous constructions define reciprocally inverse lattice isomorphisms.
    
\end{enumerate}
  \ 

    
    
\end{theorem}

\begin{proof} 
(1) Suppose $\theta \in Con(\PLA)$ and let $C_\theta$ and $\{\theta_{ii}\}_{i\in I}$ be as in Definition \ref{def: cong1}. 
We show that the pair $(C_{\theta},\{\theta_{ii}\}_{i\in I})$ satisfies conditions (i)-(iv) in Definition \ref{def: congsist}.\\
\noindent
(i) follows from Lemma \ref{lemma: congruence6}-(i). \\
\noindent
(ii) follows easily by the assumption that $\theta$ is a congruence on $\PLA$. \\
                 \noindent      (iii)-(iv) follow by Lemma \ref{lemma: congruence4}. \\
          
\noindent (2) Suppose $(C, \{\theta_{ii}\}_{i\in I}\})$ is a congruence on $\mathbb{A}$. We preliminarily prove that
\begin{equation}\label{eq: claim prova Th.8}
  \text{ for every } (i_1,j_1), (i_2, j_2) \in S_C,  \theta_{i_1j_1} \subseteq (p_{i_1,i_1 \vee i_2} \times p_{j_1, j_1 \vee j_2})^{-1}(\theta_{i_1 \vee i_2,j_1 \vee j_2}).  \tag{$\ast$}
\end{equation}
 Indeed: 
\begin{align*}
     \theta_{i_1j_1} &=(p_{i_1i_1\vee j_1} \times p_{j_1i_1 \vee j_1})^{-1}(\theta_{i_1\vee j_1, i_1 \vee j_1}) & (\text{definition}) \\
     & \subseteq (p_{i_1i_1\vee j_1} \times p_{j_1i_1 \vee j_1})^{-1}((p_{i_1 \vee j_1, i_1 \vee j_1 \vee i_2 \vee j_2} \times {p_{i_1 \vee j_1, i_1 \vee j_1 \vee i_2 \vee j_2}})^{-1}(\theta_{i_1 \vee i_2 \vee j_1 \vee j_2, i_1 \vee i_2 \vee j_1 \vee j_2})) \\
     & = (p_{i_1,i_1 \vee i_2 \vee j_1 \vee j_2} \times p_{j_1,i_1 \vee i_2 \vee j_1 \vee j_2})^{-1}(\theta_{i_1 \vee i_2 \vee j_1 \vee j_2, i_1 \vee i_2 \vee j_1 \vee j_2}) \\
     & =(p_{i_1,i_1 \vee i_2} \times p_{j_1, j_1 \vee j_2})^{-1}((p_{i_1 \vee i_2, i_1 \vee i_2 \vee j_1 \vee j_2} \times p_{j_1 \vee j_2, i_1 \vee i_2 \vee j_1 \vee j_2})^{-1}(\theta_{i_1 \vee i_2 \vee j_1 \vee j_2, i_1 \vee i_2 \vee j_1 \vee j_2})) \\
     &=(p_{i_1,i_1 \vee i_2} \times p_{j_1, j_1 \vee j_2})^{-1}(\theta_{i_1 \vee i_2,j_1 \vee j_2}),
\end{align*}
where we have simply applied the properties of composition between transition morphisms. 
We can now prove that $\theta\in Con(\PLA)$. Reflexivity follows immediately from the fact that $S_C$ and $\theta_{ii}$ (for every $i\in I$) are reflexive. To see that $\theta$ is symmetric, let $\displaystyle (a,b) \in \theta=\bigcup_{(i,j) \in S_C}{\theta_{ij}}$, then there exists $(i,j)\in S_C$ such that $(a,b) \in \theta_{ij}=(p_{ii\vee j} \times p_{ji\vee j})^{-1}(\theta_{i\vee j, i\vee j})$. Therefore $(p_{iivj}(a), p_{ji\vee j}(b)) \in \theta_{i\vee j, i \vee j}$, so by the symmetry of $\theta_{i\vee j, i \vee j}$ we also have $(p_{ji\vee j}(b), p_{ii\vee j}(a)) \in \theta_{i\vee j, i \vee j}$, that is $(b,a) \in (p_{ji\vee j} \times p_{ii\vee j})^{-1}(\theta_{i\vee j, i \vee j})=\theta_{ji} \subseteq \theta$.  \\
\noindent
To see that $\theta$ is transitive, let $(a,b),(b,c) \in \theta$; then there exist $(i,j),(j,k) \in S_C$ such that $(a,b)\in \theta_{ij}, (b,c) \in \theta_{jk}$. The fact that $(i,j),(j,k) \in S_C $ implies, by Lemma \ref{lemma: congruence5}, $ (i,i \vee k),(k,i \vee k) \in S_C$ and $(i,i\vee j \vee k), (j, i \vee j \vee k), (k, i \vee j \vee k) \in S_C$. By applying (\ref{eq: claim prova Th.8}), we have $\theta_{ij} \subseteq (p_{ii\vee j\vee k} \times p_{ji\vee j \vee k})^{-1}(\theta_{i\vee j \vee k, i \vee j \vee k})$ and $\theta_{jk} \subseteq (p_{ji\vee j\vee k} \times p_{ki\vee j \vee k})^{-1}(\theta_{i\vee j \vee k, i \vee j \vee k})$, therefore $(p_{ii\vee j \vee k}(a), p_{ji\vee j \vee k}(b)), (p_{ji\vee j \vee k}(b), p_{ki\vee j \vee k}(c)) \in \theta_{i\vee j \vee k, i \vee j \vee k} $ and (using the fact that $\theta_{i\vee j \vee k, i \vee j \vee k} $ is a congruence) $ (a,c) \in (p_{ii\vee j \vee k} \times p_{ki\vee j \vee k})^{-1}(\theta_{i\vee j \vee k, i \vee j \vee k})\neq\emptyset$, 
morevoer (composing the transition morphisms) $(p_{ii\vee j \vee k} \times p_{ki\vee j \vee k})^{-1}(\theta_{i\vee j \vee k, i \vee j \vee k}) = (p_{ii\vee k} \times p_{ki\vee k})^{-1}((p_{i\vee k,i\vee j \vee k} \times p_{i\vee k, i \vee j \vee k})^{-1}(\theta_{i\vee j \vee k, i \vee j \vee k})) = (p_{ii\vee k} \times p_{ki\vee k})^{-1}(\theta_{i\vee k, i \vee k})$ (where the last equality follows by Definition \ref{def: congsist} and the fact $(i\vee k, i\vee j\vee k)\in S_{C}$). Consequently $ (p_{ii\vee k} \times p_{ki\vee k})^{-1}(\theta_{i\vee k, i \vee k})\neq\emptyset $, so 
$(i,k) \in S_C$ and $\theta_{ik}=(p_{ii\vee j \vee k} \times p_{ki\vee j \vee k})^{-1}(\theta_{i\vee j \vee k, i \vee j \vee k})$, thus $(a,c) \in \theta_{ik} \subseteq \theta$. \\
\noindent
To see that $\theta$ preserves arbitrary operations, let $(a_1,b_1) \in \theta_{i_1j_1}, \ldots,(a_n,b_n) \in \theta_{i_nj_n}$ and $f $ be an $n-$ary operation. 
Let $i=i_1 \vee \ldots \vee i_n$ and $j=j_1 \vee \ldots \vee j_n$, then $(i,j) \in S_C$. Using (\ref{eq: claim prova Th.8}), we have that for any $ k\in \{1, \ldots, n\}$: $\theta_{i_1j_1}\subseteq (p_{i_ki} \times p_{j_kj})^{-1}(\theta_{ij})$, so $(p_{i_ki}(a_k), p_{j_kj}(b_k)) \in \theta_{ij}=(p_{ii\vee j} \times p_{ji\vee j})^{-1}(\theta_{i\vee j, i \vee j})$, that is $\forall k\in \{1, \ldots, n\}: (p_{i_k,i\vee j}(a_k), p_{j_k,i\vee j}(b_k)) \in \theta_{i\vee j, i \vee j}$. Since $\theta_{i\vee j, i \vee j} \in Con(\A_{i\vee j})$, we deduce 
$(p_{i,i\vee j} \times p_{j,i\vee j})(f^{\A_i}(p_{i_1i}(a_{i_1}), \ldots, p_{i_ni}(a_{i_n})),f^{\A_j}(p_{j_1j}(b_1),...,p_{j_nj}(b_n)))= $ \\
$= (f^{\A_{i\vee j}}(p_{i,i\vee j}(p_{i_1,i}(a_1)), \ldots, p_{i,i\vee j}(p_{i_n,i}(a_n)), f^{\A_{i\vee j}}(p_{j,i\vee j}(p_{j_1,j}(b_1)), \ldots, p_{j,i\vee j}(p_{j_n,j}(b_n))))=$\\$(f^{\A_{i\vee j}}(p_{i_1,i\vee j}(a_1), \ldots, p_{i_n,i\vee j}(a_n), f^{\A_{i\vee j}}(p_{j_1,i\vee j}(b_1), \ldots, p_{j_n,i\vee j}(b_n))) \in \theta_{i \vee j, i \vee j}$, therefore, by Definition, we have that $(f^{\A_i}(p_{i_1i}(a_{i_1}), \ldots, p_{i_ni}(a_{i_n})),f^{\A_j}(p_{j_1j}(b_1),...,p_{j_nj}(b_n))) \in (p_{i,i\vee j} \times p_{j,i \vee j})^{-1}(\theta_{i\vee j, i \vee j})=\theta_{ij}$, that is $(f^{\PLA}(a_1,\ldots,a_n), f^{\PLA}(b_1, \ldots, b_n)) \in \theta_{ij}\subseteq \theta$.

\noindent (3) If $\theta \in Con(\PLA)$, then (by Lemma \ref{lemma: congruence4}, item (iii)) $\theta_{ij}=(p_{ii\vee j} \times p_{ji\vee j})^{-1}(\theta_{i\vee j,i\vee j})$ for every $(i,j) \in S_{\theta}=S_{C_\theta}$ (this equality is a consequence of Lemma \ref{lemma: congruence6}, item (ii) and the second part of Lemma \ref{lemma: congruence5}, item (iii)), and 
            $$\displaystyle \theta= \theta \cap (A\times A)=\theta \cap \bigcup_{(i,j)\in I\times I}{(A_i \times A_j)} =\bigcup_{(i,j) \in I \times I}{[\theta \cap (A_i \times A_j)]}=\bigcup_{(i,j) \in S_{\theta}}{\theta_{ij}}.$$

 \noindent The converse follows by the first part of Lemma \ref{lemma: congruence5}, item (iii).           
\end{proof}



As previously stated, a congruence on $\PLA$ is uniquely determined by the semilattice $S_\theta$ (or by the congruence $C_\theta$, since Lemma \ref{lemma: congruence6} tells us we can always obtain one from the other) and its pure fibers $\{\theta_{ii}\}_{i\in I}$.

\begin{corollary}\label{cor: corollary1}
    Let $\theta, \theta' \in Con(\PLA)$. Then $\theta=\theta'$ if and only if $S_{\theta}=S_{\theta'}$ and $ \theta_{ii}=\theta'_{ii}$ for every $i\in I$.
\end{corollary}


The description of congruences for algebras in a regularized variety $R(\V)$ can be fruitfully applied when $\V$ is an \emph{ideal determined} variety, as shown in the following.
    \begin{remark}
        Let $\V$ be a variety whose language contains a constants symbol $0$. Recall that a term $p(\vec{x}, \vec{y})$, where $\vec{x}=(x_1, \ldots, x_n)$ and $\vec{y}=(y_1, \ldots, y_m)$, is an \emph{ideal term in} $\vec{y}$ for $\mathcal{V}$ if $\V \models p(\vec{x}, \vec{0}) \approx 0 $ (with $\vec{0} = (0,\dots, 0)$), and that a non-empty subset $I$ of $\A \in \V$ is an \emph{ideal} of $\A$ if we have $p(\vec{a}, \vec{i}) \in I$ for every $\vec{a} \in A^{n},\ \vec{i} \in I^m$ and $p(\vec{x}, \vec{y})$ ideal term in $\vec{y}$ for $\V$. Observe the previous definition is not requiring the variables $x$, $y$ to actually occur in $p$, therefore the constant $0$ is always an ideal term in $\vec{y}$ for $\V$, hence given a congruence $\theta$ on $\A \in \V$ it holds that $[0]_{\theta}$ is an ideal of $\A$. The converse is not true in general, i.e. an ideal $J$ of $\A$ is not necessary the zero-class of a (unique) congruence $\theta\in Con(\A)$.        
     We say that $\V$ is \emph{ideal-determined} if for every algebra $\A \in \V$ and for every ideal $I$ of $\A$ there exists a unique $\theta \in Con(\A)$ such that $I=[0]_{\theta}$. In this case, the map $Con(\A) \rightarrow Id(\A),\ \theta \mapsto [0]_{\theta}$ is a lattice isomorphism. A fundamental result from Gumm and Ursini \cite{GummUrsini, Ursini} establishes that $\V$ is ideal-determined if and only if $\mathcal{V}$ is \emph{$0$-regular} (i.e. for every $\A \in \mathcal{V}$ and for every $\theta, \theta' \in Con(\A)$ if $[0]_{\theta}=[0]_{\theta'}$ then $\theta=\theta'$) and \emph{subtractive} (i.e. there exists a binary term $s(x,y)$, called \emph{subtraction term}, such that $\V\models s(x, 0) \approx x$ and $\V\models s(x,x) \approx 0$). Observe that any subtractive variety $\V$ (hence any ideal-determined variety $\V$) is strongly irregular, since $\V\models s(x, s(y,y))\approx x$, for any $s$ subtraction term for $\V$. Therefore regularizations of ideal-determined varieties fall under the scope of Theorem \ref{theorem: congruence5} and 
        Corollary \ref{cor: corollary1}. Notice that (by regularity) for $R(\V)$ the constant term $0$ is (up to identities) the only ideal term, so ideals for an algebra $\A \in R(\V)$ are just subsets $S \subseteq A$ such that $0^{\A} \in S$. However, at light of Theorem \ref{theorem: congruence5}, it makes sense to introduce a more restrictive notion of ideal for an algebra $\A$ in $R(\V)$ and establish an isomorphism between $Con(\A)$ and the lattice of ideals of $\A$ (ordered by inclusion). Let $\A$ be a P\l onka sum over a semilattice direct system $(\{\A_i\}_{i\in I}, (I, \le), \{p_{ij}\}_{i \le j})$, where $\V$ is an ideal-determined variety, then we say that a \emph{P\l onka ideal} for $\A$ is a pair $(C, \{J_i\}_{i\in I})$, where $C \in Con(\I)$ and $J_i$ is an ideal of $\A_i$ for every $i\in I$ such that $J_i \subseteq (p_{i,i\vee j})^{-1}(J_{i\vee j})$ for every $i,j \in I$ and $J_i=(p_{i,i\vee j})^{-1}(J_{i\vee j})$ if $(i,j) \in C$. We denote by $Id(\A)$ the lattice of P\l onka ideals of $\A$ ordered by $(C, \{J_i\}_{i\in I})\leq (C', \{K_i\}_{i\in I})$ iff $C\subseteq C'$ and $J_i\subseteq K_i$, for every $i\in I$. Theorem \ref{theorem: congruence5} and Corollary \ref{cor: corollary1} establish that the map $\theta \mapsto (S_{\theta}, [0]_{\theta_{ii}})_{i\in I}$ is an isomorphism between $Con(\A)$ and $Id(\A)$.
    \end{remark}


The previous remark allows us to characterize congruences for several regularization of strongly irregular varieties. 

\begin{example}
Using Theorem \ref{theorem: congruence5}, we can characterize the congruences of Clifford semigroups. Let $\mathbb{C}=(\{\mathbf{G}_i\}_{i\in I}, (I, \le), \{p_{ij}\}_{i \le j})$ be a semilattice direct system of groups and $\mathbf{C}$ its P\l onka sum, i.e. a Clifford semigroup. Given $\theta \in Con(\mathbf{C})$, we have that $S_{\theta} = \{(i,j)\in I\times I\; |\; \theta\cap (G_{i}\times G_j)\neq\emptyset \}$ is a congruence on the semilattice $\I$ (by Lemma \ref{corollary: lemma5}) 
and $\theta_{ii} = \theta\cap (G_{i}\times G_{i})$ is a congruence on each group $\G_{i}$, therefore we can identify $\theta_{ii}$ with the correspondent normal subgroup $N_{i}$ of $\G_{i}$ (recall that $N_{i} = [e_{i}]_{\theta_{ii}}$, with $e_i$ the neutral element of $\G_{i}$); moreover, $N_{i}\subseteq p_{i,i\vee j}^{-1}(N_{i\vee j})$, for every $i,j\in I$, and, in particular, $N_{i}= p_{i,i\vee j}^{-1}(N_{i\vee j})$, for $(i,j)\in S_{\theta}$. Theorem \ref{theorem: congruence5} states that we can identify $\theta$ with $(S_{\theta}, \bigcup_{i\in I}{N_i})$. \\
\noindent
Conversely, let $\{N_i\}_{i\in I}$ be a family of normal subgroups (one for each fiber of $\mathbf{C}$) and $S \in Con(\I)$ such that for all $i,j\in I: N_i \subseteq p_{ii\vee j}^{-1}(N_{i\vee j})$ and $N_i = p_{ii\vee j}^{-1}(N_{i\vee j})$ in case $(i,j) \in S$. Then letting 
$$\displaystyle N=\bigcup_{i\in I}{N_i}$$ 
\noindent
and defining $\theta_{N,S}$ as: $ \forall a\in G_i, b\in G_j: (a,b) \in \theta_{N,S} \iff (i,j) \in S \text{ and } ab^{-1} \in N$ we have that $\theta_{N,S} \in Con(\mathbf{C})$.

\end{example}

\begin{example}
    Using Theorem \ref{theorem: congruence5}, we can characterize the congruences of Involutive bisemilattices. Let $\mathbb{B}=(\{\mathbf{B}_i\}_{i\in I}, (I, \le), \{p_{ij}\}_{i \le j})$ be a semilattice direct system of boolean algebras and $\B$ its P\l onka sum, i.e. an involutive bisemilattice. Given $\theta \in Con(\mathbf{B})$, we have that $S_{\theta} = \{(i,j)\in I\times I\; |\; \theta\cap (B_{i}\times B_j)\neq\emptyset \}$ is a congruence on the semilattice $\I$ (by Lemma \ref{corollary: lemma5}) 
and $\theta_{ii} = \theta\cap (B_{i}\times B_{i})$ is a congruence on each boolean algebra $\B_{i}$, therefore we can identify $\theta_{ii}$ with the correspondent filter $F_{i}$ of $\G_{i}$ (recall that $F_{i} = [1_{i}]_{\theta_{ii}}$, with $1_i$ the top element of $\B_{i}$); moreover, $F_{i}\subseteq p_{i,i\vee j}^{-1}(F_{i\vee j})$, for every $i,j\in I$, and, in particular, $F_{i}= p_{i,i\vee j}^{-1}(F_{i\vee j})$, for $(i,j)\in S_{\theta}$. Theorem \ref{theorem: congruence5} states that we can identify $\theta$ with $(S_{\theta}, \bigcup_{i\in I}{F_i})$. \\
\noindent
Conversely, let $\{F_i\}_{i\in I}$ be a family of filters (one for each fiber of $\B$) and $S \in Con(\I)$ such that for all $i,j\in I: F_i \subseteq p_{ii\vee j}^{-1}(F_{i\vee j})$ and $F_i = p_{ii\vee j}^{-1}(F_{i\vee j})$ in case $(i,j) \in S$. Then letting 
$$\displaystyle F=\bigcup_{i\in I}{F_i}$$ 
\noindent
and defining $\theta_{N,S}$ as: $ \forall a\in B_i, b\in B_j: (a,b) \in \theta_{F,S} \iff (i,j) \in S \text{ and } a \leftrightarrow b \in F$ we have $\theta_{F,S} \in Con(\B)$.
\end{example}

\begin{example}
Using Theorem \ref{theorem: congruence5}, we can characterize the congruences of dual weak braces. Let $\mathbb{S}=(\{\mathbf{S}_i\}_{i\in I}, (I, \le), \{p_{ij}\}_{i \le j})$ be a semilattice direct system of skew-braces and $\mathbf{D}$ its P\l onka sum, i.e. a dual weak brace (see \cite{Stefanelli25,CatinoMediterranean}). Given $\theta \in Con(\mathbf{D})$, we have that $S_{\theta} = \{(i,j)\in I\times I\; |\; \theta\cap (S_{i}\times S_j)\neq\emptyset \}$ is a congruence on the semilattice $\I$ (by Lemma \ref{corollary: lemma5}) 
and $\theta_{ii} = \theta\cap (S_{i}\times S_{i})$ is a congruence on each skew-brace $\mathbf{S}_{i}$, therefore we can identify $\theta_{ii}$ with the correspondent (Gumm-Ursini) ideal $J_{i}$ of $\mathbf{S}_{i}$ (recall that $J_{i} = [0_{i}]_{\theta_{ii}}$, with $0_i$ the neutral element of $\mathbf{S}_{i}$); moreover, $J_{i}\subseteq p_{i,i\vee j}^{-1}(J_{i\vee j})$, for every $i,j\in I$, and, in particular, $J_{i}= p_{i,i\vee j}^{-1}(J_{i\vee j})$, for $(i,j)\in S_{\theta}$. Theorem \ref{theorem: congruence5} states that we can identify $\theta$ with $(S_{\theta}, \bigcup_{i\in I}{N_i})$. \\
\noindent
Conversely, let $\{J_i\}_{i\in I}$ be a family of ideals (one for each fiber of $\mathbf{C}$) and $S \in Con(\I)$ such that for all $i,j\in I: J_i \subseteq p_{ii\vee j}^{-1}(J_{i\vee j})$ and $J_i = p_{ii\vee j}^{-1}(J_{i\vee j})$ in case $(i,j) \in S$. Then letting 
$$\displaystyle J=\bigcup_{i\in I}{J_i}$$ 
\noindent
and defining $\theta_{J,S}$ as: $ \forall a\in S_i, b\in S_j: (a,b) \in \theta_{J,S} \iff (i,j) \in S \text{ and } ab^{-1} \in J$ we have that $\theta_{J,S} \in Con(\mathbf{C})$.

\end{example}




With the characterization of congruences ad hand, we can also describe the structure of the P\l onka sum decomposition of a quotient of $\A=\PLA$ (under the same assumption of the section, i.e. $\A\in R(\V)$ with $\V$ a strongly irregular variety).


\begin{proposition}
    Let $\theta \in Con(\A)$. Then $\mathbf{A}/\theta$ is (isomorphic to) the P\l onka sum over the semilattice direct system given by:  
    
    \begin{enumerate}
        \item the semilattice $\mathbf{I}/C_{\theta}$;
        \item for every $i\in I$, the universe of the fiber is $\left (A/\theta \right )_{[i]_{{C_\theta}}}= \bigl\{[a]_{\theta} \mid j \in [i]_{C_{\theta}},\ a \in A_j\bigr\} $;
        \item the transition homomorphisms are, for every $ i, k\in I $ such that $ [i]_{C_\theta} \le [k]_{C_\theta} $, 
        $$ p_{[i]_{C_\theta}[k]_{C_\theta}}([a]_{\theta})=[p_{jj\vee k}(a)]_{\theta},$$
    for any $ j\in [i]_{C_\theta} $ and $ a \in A_j$.      
            \end{enumerate}
\end{proposition}

\begin{proof}
Let $ i,j\in I $ and $ a\in A_i, b \in A_j$. 
Then $ [a]_{\theta}=[a]_{\theta} \cdot^{\A/\theta}[b]_{\theta}=[a\cdot^{\A} b]_{\theta}=[p_{ii\vee j}(a)]_\theta $ and $
              [b]_{\theta}=[b]_{\theta} \cdot^{\A/\theta}[a]_{\theta}=[b\cdot^{\A} a]_{\theta}=[p_{ji\vee j}(b)]_{\theta} $ if and only if 
              $    (a, p_{ii\vee j}(a)) \in \theta $ and $(b, p_{ji\vee j}(b)) \in \theta $. By Definition of $S_{\theta}$ and Lemma \ref{lemma: congruence1}, this is equivalent to $  \{(i, i \vee j), (j,i \vee j)\} \subseteq S_\theta $, i.e $ (i,j) \in C_{\theta}$.
\noindent The other claims follow immediately from the previous fact and from the general theory of P\l onka sums.  
\end{proof}

\subsection*{Generated congruences}
\noindent In this subsection, we provide a description of generated congruences in a P\l onka sum $\A=\PLA$. 

We will use the same notation of Definition \ref{def: cong1} extended from congruences to arbitrary (binary) relations, i.e. for instance by $R_{ij}$ we will here mean $R\cap (A_i\times A_j)$, with $R$ a relation on $A$.  


\begin{theorem} \label{th: congruenza generata}
    Let $R \subseteq A \times A$, $C \in Con(I)$ such that $S_R \subseteq C$ and $\Psi_{i}:= Cg^{\A_i} \displaystyle\left (\bigcup_{u,v \le i}(p_{ui} \times p_{vi})(R_{uv}) \right )$. 
    For every $i\in I$, let $\theta_{ii}$ a binary relation on $A_i$ defined as $(a,b) \in \theta_{ii} $ if and only if the following are satisfied: 
    \begin{enumerate}
    \item there exists $ k \in I $ such that $i \le k $ and $ (i,k) \in S_C = \{(i,j)\in C \;|\; (p_{ii\vee j} \times p_{ji\vee j})^{-1}(\Psi_{i\vee j}) \neq \emptyset \}$;
     
    \item $(p_{ik}(a), p_{ik}(b)) \in\Psi_{k}  $. 
    \end{enumerate}
 \noindent   
    Then $\theta_{ii} \in Con(\A_i)$ and $(S_C, \{\theta_{ii}\}_{i\in I})$ is a congruence on the direct system $\mathbb{A}$ whose associated congruence $\theta$ (see Theorem \ref{theorem: congruence5}) is such that $R \subseteq \theta$. In particular, for every $R \subseteq A\times A$ and for every $C \in Con(\I)$ there exists $\theta \in Con(\A)$ such that $C_{\theta}=C$. 
\end{theorem}

\begin{proof}
Preliminarily observe that $S_{C}$ is well defined (see Lemma \ref{lemma: congruence5}) with respect to the family of congruences $\Psi_{i}$, and Lemma \ref{lemma: congruence5} assures that $S_C$ is a reflexive, symmetric and upper transitive subsemilattice of $\I\times \I$. We now check that $\theta_{ii}$ is a congruence on $\A_{ii}$. \\
\noindent
\emph{Reflexivity}. Let $a \in A_i$, then $ \displaystyle (a,a)=(p_{ii}(a), p_{ii}(a)) \in \Delta_i \subseteq \Psi_i $ 
, so $(a,a) \in \theta_{ii}$. \\
\noindent
\emph{Symmetry}. Let $(a,b) \in \theta_{ii}$, that is there exist $k\in I$ such that $i \le k$, $(i,k) \in S_C$ and $\displaystyle (p_{ik}(a), p_{ik}(b)) \in \Psi_{k} $, therefore $\displaystyle (p_{ik}(b), p_{ik}(a)) \in \Psi_{k}$, so $(b,a) \in \theta_{ii}$.\\
    \noindent 
\emph{Transitivity}. Let $(a,b), (b,c) \in \theta_{ii}$. Then, by definition, there exist $k_1, k_2 \in I$ such that $i \le k_1, k_2, (i,k_1), (i, k_2) \in S_C$, and $\displaystyle (p_{ik_1}(a), p_{ik_1}(b)) \in \Psi_{k_1} $, $(p_{ik_2}(b), p_{ik_2}(c)) \in \Psi_{k_{2}} $.  
From the fact that $ (p_{k_1k_1 \vee k_2} \times p_{k_1k_1 \vee k_2})(\Psi_{k_1}) 
      \subseteq \Psi_{k_{1}\vee k_{2}}$ and $ (p_{k_2 k_1 \vee k_2} \times p_{k_2k_1 \vee k_2})(\Psi_{k_2}) 
      \subseteq \Psi_{k_{1}\vee k_{2}}$ and the transitivity of $\Psi_{k_{1}\vee k_{2}}$ we deduce 
$ ((p_{i,k_1 \vee k_2})(a), (p_{i,k_1 \vee k_2})(c)) \in \Psi_{k_1 \vee k_2} $. Furthermore, we have $(i, k_1 \vee k_2) \in S_{C}$ (since $S_{C}$ is a subsemilattice of $\I\times \I$), hence $(a,c) \in \theta_{ii}$.\\
\noindent 
\emph{Compatibility with the operations}: Let $(a_1,b_1),\ldots,(a_n,b_n) \in \theta_{ii}$. Then, by definition, there exist $k_1,\ldots, k_n \in I$ such that $i \le k_1, \ldots, k_n$, $(i,k_1), \ldots, (i, k_n) \in S_C$ and for every $ j\in \{1,\ldots,n\}$, $ (p_{ik_j}(a_j), p_{ik_j}(b_j)) \in \Psi_{k_j}$. Since we have that for every $ j \in\{1, \ldots, n\}$, $(p_{k_jk_1 \vee \ldots  \vee k_n} \times p_{k_jk_1 \vee \ldots \vee k_n}) \Psi_{k_j}\subseteq \Psi_{k_1 \vee \ldots \vee k_n}$ then $ (p_{ik_1 \vee \ldots k_n}(a_j), p_{ik_1 \vee \ldots k_n}(b_j)) \in \Psi_{{k_1 \vee \ldots \vee k_n}}$. So 
    
    \begin{align*} \displaystyle 
         &(p_{ik_1 \vee \ldots \vee k_n}(f^{\A_i}(a_1,\ldots, a_n)), p_{ik_1 \vee \ldots \vee  k_n}(f^{\A_i}(b_1,\ldots, b_n))) =\\
         &\qquad = (f^{\A_{k_1 \vee \ldots \vee k_n}}(p_{ik_1 \vee \ldots 
 \vee k_n}(a_1),\ldots, p_{ik_1 \vee \ldots k_n}(a_n)),f^{\A_{k_1 \vee \ldots \vee k_n}}(p_{ik_1 \vee \ldots k_n}(b_1),\ldots, p_{ik_1 \vee \ldots k_n}(b_n))) \in \Psi_{k_1 \vee \ldots \vee k_n}.
     \end{align*}
    
    \noindent By definition, since $(i, k_1 \vee \ldots \vee k_n) \in S_C$, we have $(f^{\A_1}(a_1, \ldots, a_n), f^{\A_i}(b_1, \ldots, b_n)) \in \theta_{ii}$.\\

    \noindent Now we prove conditions \textit{(iii)-(iv)} of Definition \ref{def: congsist}. \\
    \noindent 
    (iii) Let $(a,b) \in \theta_{ii}$, then by definition there exists $k\in I$ such that $i \le k$, $(i,k) \in S_C$ and $(p_{ik}(a), p_{ik}(b)) \in \Psi_{{k}}$. Consequently $i \vee j \le k \vee j$, $(i \vee j, k \vee j) \in S_C$, and $(p_{i\vee j, k \vee j}((p_{ii\vee j})(a)), p_{i\vee j,k\vee j}((p_{ii\vee j})(b)))=(p_{ik\vee j}(a), p_{ik\vee j}(b))=(p_{kk\vee j}(p_{ik}(a)),p_{kk\vee j}(p_{ik}(b))) \in (p_{kk\vee j} \times p_{kk\vee j})(\Psi_{{k}})\subseteq \Psi_{k \vee j}$, therefore by definition 
    $(a,b) \in (p_{ii\vee j \times p_{ii\vee j}})^{-1}(\theta_{i\vee j, i \vee j})$.\\
    \noindent 
    (iv) Suppose $(i,j) \in S_C$. $(a,b) \in (p_{ii\vee j} \times p_{ii\vee j})^{-1}(\theta_{i\vee j, i \vee j})$ if and only if $(p_{ii\vee j}(a), p_{ii\vee j}(b)) \in \theta_{i\vee j,i \vee j}$ if and only if there exists $k \in I$ such that $i \vee j \le k$, $(i\vee j, k) \in S_C$ and $(p_{ik}(a),p_{ik}(b))=(p_{i\vee j,k}(p_{ii\vee j}(a)), p_{i\vee j,k}(p_{ii\vee j}(b))) \in \Psi_{{k}}$. That is $(a,b) \in \theta_{ii}$, since $i \le i\vee j \le k$ imply $i \le k$ and $(i, i \vee j), (i\vee j,k) \in S_C$ imply (upper transitivity) $(i,k) \in S_C$.\\
 \noindent Since for every $i \in I$ we have $i \le i, (i,i) \in S_C$ and $(p_{ii} \times p_{ii})(\Psi_{i})=\Psi_{i}$, we deduce $\Psi_{i} \subseteq \theta_{ii}$ for every $i\in I$. Finally, observe that $(i,j) \in S_R$ imply $\emptyset \neq R_{ij} \subseteq (p_{ii\vee j} \times p_{ji\vee j})^{-1}((p_{ii\vee j} \times p_{ji\vee j})(R_{ij})) \subseteq (p_{ii\vee j} \times p_{ji\vee j})^{-1}(Cg^{\A_{i\vee j}}(\bigcup_{u,v \le i \vee j}{(p_{ui\vee j} \times p_{vi\vee j})(R_{uv})}))=(p_{ii\vee j} \times p_{ji\vee j})(\Psi_{i\vee j}) \subseteq (p_{ii\vee j} \times p_{ji\vee j})(\theta_{i\vee j, i \vee j})$, so $(i,j) \in S_C$ and $R_{ij} \subseteq \theta_{ij}=(p_{ii\vee j} \times p_{ji\vee j})^{-1}(\theta_{i\vee j, i \vee j})$ (see Theorem \ref{theorem: congruence5}). Consequently $R \subseteq \theta$.
\end{proof} 


\begin{corollary}
    Let $R \subseteq A \times A$. Then for every $C \in Con(\I)$ with $S_R \subseteq C$ there exists $\theta \in Con(\A)$ such that $R \subseteq \theta$ and $C_{\theta}=C$.
\end{corollary}

\begin{proof}
    This follows immediately from Theorem \ref{th: congruenza generata} using also item (iii) of Lemma \ref{lemma: congruence5}.
\end{proof}
\begin{corollary}\label{cor: gen}
    Let $R \subseteq A \times A$ and $\Psi_{i}:= Cg^{\A_i} \displaystyle\left (\bigcup_{u,v \le i}(p_{ui} \times p_{vi})(R_{uv}) \right )$. Then $S_{Cg^{\A}(R)}=S_{Cg^{\I}(S_R)}$ and $(a,b) \in (Cg^{\A}(R))_{ii}$ if an only if the following are satisfied:
        \begin{enumerate}
    \item there exists $ k \in I $ such that $i \le k $ and $ (i,k) \in S_C = \{(i,j)\in C \;|\; (p_{ii\vee j} \times p_{ji\vee j})^{-1}(\Psi_{i\vee j}) \neq \emptyset \}$;
     
    \item $(p_{ik}(a), p_{ik}(b)) \in\Psi_{k}  $. 
    \end{enumerate} 
\end{corollary}
\begin{proof}
    Let $\theta \in Con(\A)$ such that $R \subseteq \theta$, than $S_R \subseteq S_{\theta} \subseteq C_{\theta}$, so $Cg^{\mathbf{I}}(S_R) \subseteq C_{\theta}$, therefore $S_{Cg^{\mathbf{I}}(S_R)} \subseteq S_{C_\theta}=S_{\theta}$. Moreover, for every $i\in I$ let $\eta_{ii} \subseteq A_i \times A_i$ be the congruence on $\A_i$ defined in Theorem \ref{th: congruenza generata}, than for every $i\in I$ and for every $(a,b) \in A_i \times A_i$ we have $(a,b) \in \eta_{ii}$ if and only if there exists $k \in I$ with $i \le k,\ (i,k) \in S_{Cg^{\mathbf{I}}(R)} \subseteq S_{\theta}$ and such that $(p_{ik}(a), p_{jk}(b) \in \Psi_{{k}} \subseteq \theta_{kk}$, so $(p_{ik}(a),p_{ik}(b)) \in \theta$. But also $(i,k) \in S_{\theta}$, then $(a, p_{ik}(a)), (b, p_{ik}(b))\in \theta$, and so $(a,b) \in \theta_{ii}$ by transitivity of $\theta$.
\end{proof}

\subsection*{Factor congruences}

\noindent Recall that for an algebra $\A$ a \emph{pair of factor congruences} is a pair $(\theta, \theta^{\ast}) \in Con(\A) \times Con(\A)$ such that $\theta \cap \theta^{\ast}=\Delta_{\A},\ \theta \vee \theta^{\ast}=\nabla_{\A}$ and $\theta, \theta^{\ast}$ are permutable, i.e. $\theta \circ \theta^{\ast}=\theta^{\ast} \circ \theta$ (see \cite[Definition 7.4]{BuSa00}). In this subsection, we characterize the factor congruences of a P\l onka sum. 

Recall that if $\theta_1$ and $\theta_2$ are congruences on an algebra $\A$, their relational product $\theta_1 \circ \theta_2$ is a congruence on $\A$ if and only if $\theta_{1}$, $\theta_{2}$ are permutable (in such a case $\theta_{1}\circ\theta_{2} = \theta_{1}\vee\theta_{2}$). It is natural to wonder what is the structure of the congruence $\theta_1 \circ \theta_2$, when $\A$ is a P\l onka sum over a semilattice direct system of algebras in a strongly irregular variety $\V$  

\begin{lemma}\label{lemma: int}
    Let $\A$ be a P\l onka sum over a semilattice direct system $\mathbb{A}$ of algebras in a strongly irregular variety $\V$ and $\theta_1, \theta_2 \in Con(\A)$. Then: 
    \begin{enumerate}
        \item $C_{\theta_1\cap\ \theta_2}=C_{\theta_1} \cap C_{\theta_2}$;
        \item $(\theta_1 \cap \theta_2)_{ii}=(\theta_1)_{ii} \cap (\theta_2)_{ii}$ for every $i\in I$.
    \end{enumerate}
\end{lemma}
\begin{proof}
(1) It is routine to verify that $S_{\theta_1 \cap\ \theta_2} \subseteq S_{\theta_1} \cap S_{\theta_2}$. Let $(i,j) \in C_{\theta_1 \cap\ \theta_2}$, then $\{(i,i \vee j), (j, i \vee j)\} \subseteq S_{\theta_1 \cap\ \theta_2} \subseteq S_{\theta_1} \cap S_{\theta_2}$  
, so $(i,j) \in C_{\theta_1} \cap C_{\theta_2}$. For the converse, let $(i,j) \in C_{\theta_1} \cap C_{\theta_2}$, then $\{(i, i \vee j), (j, i \vee j)\} \subseteq S_{\theta_1} \cap S_{\theta_2}$. Let $a \in A_i,\ b \in A_j$, then $\{(a, p_{ii\vee j}(a)), (b, p_{ji\vee j}(b))\} \subseteq \theta_1 \cap \theta_2$, therefore $\{(i,i\vee j), (j, i \vee j)\} \subseteq S_{\theta_1 \cap\ \theta_2}$, that is $(i,j) \in C_{\theta_1 \cap\ \theta_2}$.\\
\noindent
(2) $(\theta_1 \cap \theta_2)_{ii}=(\theta_1 \cap \theta_2) \cap (A_i \times A_i)=(\theta_1 \cap (A_i \times A_i)) \cap (\theta_2 \cap (A_i \times A_i))=(\theta_1)_{ii} \cap (\theta_2)_{ii}$ for every $i\in I$.
\end{proof}


\begin{lemma}\label{lemma: Fattore1}
    Let $\A$ be a P\l onka sum over a semilattice direct system $\mathbb{A}$ of algebras in $\V$ and $\theta_1, \theta_2 \in Con(\A)$ two permutable congruences. Then $S_{\theta_{1} \circ \theta_{2}} \subseteq S_{\theta_1} \circ S_{\theta_2}$, $C_{\theta_1} \circ C_{\theta_2} \subseteq C_{\theta_1 \circ \theta_2},$ and, for every $i\in I$
     $$\displaystyle (\theta_1 \circ \theta_2)_{ii}= \bigcup_{(i,j) \in S_{\theta_1} \cap S_{\theta_2}}{((\theta_1)_{ij} \circ (\theta_2)_{ji})}.$$     
\end{lemma}

\begin{proof}
Suppose $(i,k) \in S_{\theta_1 \circ \theta_2}$. Then by definition there exists $(a,c) \in A_i \times A_k$ such that $(a,c) \in \theta_1 \circ \theta_2$, i.e. there exist $j \in I$ and $b \in A_j$, for some $j\in I$, such that $(a,b) \in \theta_1$ and $(b,c) \in \theta_2$. Therefore, $(i,j) \in S_{\theta_1}$ and $(j,k) \in S_{\theta_2}$, which implies $(i,k) \in S_{\theta_1 \circ \theta_2}$.\\
\noindent 
Suppose $(i,k) \in C_{\theta_1} \circ C_{\theta_2}$. Then there exists $j \in I$ such that $(i,j) \in C_{\theta_1}$ and $(j,k) \in C_{\theta_2}$, that is $\{(i, i \vee j), (j, i \vee j)\} \subseteq S_{\theta_1}$ and $\{(j,j \vee k), (k,j\vee k)\} \subseteq S_{\theta_2}$. So from preservation of join we also have $(i \vee j, i \vee j \vee k) \in S_{\theta_2}$. Moreover, for every $ a \in A_i $, $ (a, p_{ii\vee j}(a)) \in \theta_1$ and $(p_{ii\vee j}(a), p_{ii\vee j \vee k}{(a))} \in \theta_2$, thus $(a, p_{ii\vee j \vee k}(a)) \in \theta_1 \circ \theta_2 $ which implies $ (i, i \vee j \vee k) \in S_{\theta_1 \circ \theta_2}$. By symmetry and permutability of $\theta_1$ and $\theta_2$ we also get $(k, i \vee j \vee k) \in S_{\theta_1 \circ \theta_2}$. Consequently, by upper transitivity of $S_{\theta_1 \circ \theta_2}$, we have $\{(i, i \vee k), (k,i \vee k)\} \subseteq S_{\theta_1 \circ \theta_2} $, proving that $ (i,k) \in C_{\theta_1 \circ \theta_2}$.

    \noindent Finally, suppose $(a,c) \in (\theta_1 \circ \theta_2)_{ii}$. Then, by definition, $(a,c) \in (A_i \times A_i) \cap (\theta_1 \circ \theta_2)$. This means $(a,c) \in A_i \times A_i$ and there exists $b \in A_j$ such that $(a,b) \in \theta_1$ and $(b,c) \in \theta_2$. From this, we deduce that there exists $j \in I$ such that $(a,c) \in (\theta_1)_{ij} \circ (\theta_2)_{ji}$. Conversely, suppose for some $j \in I$, we have $(a,c) \in (\theta_1)_{ij} \circ (\theta_2)_{ji}$. Then, by definition, there exists $b \in A_j$ such that $(a,b) \in (\theta_1)_{ij} \subseteq \theta_1$ and $(b,c) \in (\theta_2)_{ji} \subseteq \theta_2$. Therefore, $(a,c) \in \theta_1 \circ \theta_2$, and thus $(a,c) \in (\theta_1 \circ \theta_2)_{ii}$. 
\end{proof}

\noindent From the preceding, we deduce the following characterization for \emph{factor congruences}.

\begin{theorem}\label{th: congruenze fattore}
    Let $\A$ be a P\l onka sum over a semilattice direct system $\mathbb{A}$ of algebras in $\V$ and $\theta_1, \theta_2 \in Con(A)$. Then the following are equivalent: 
    \begin{enumerate}
        \item $(\theta_1, \theta_2)$ is a pair of factor congruences on $\A$; 
        \item $\theta_{1}, \theta_{2}$ are permutable, $(C_{\theta_1}, C_{\theta_2})$ is a pair of factor congruences on $\mathbf{I}$ and for every $ i\in I $ $ ((\theta_1)_{ii}, (\theta_2)_{ii})$ is a pair of factor congruences on $\A_i$.
    \end{enumerate}
    \end{theorem}
\begin{proof}
(1) $(\Rightarrow) $ (2). Suppose $(\theta_1, \theta_2)$ is a pair of factor congruences on $\A$. Then, in particular, $\theta_1 \cap \theta_2 = \Delta_A$, from which we deduce by Lemma \ref{lemma: int}-(1) that 
$C_{\theta_1} \cap C_{\theta_2} = 
C_{\theta_1 \cap \theta_2} = C_{\Delta_{\A}}=\Delta_{\mathbf{I}}$. Furthermore, $\nabla_{\mathbf{I}} = S_{\nabla_A}= S_{\theta_1 \vee \theta_2} = S_{\theta_1 \circ \theta_2} \subseteq  S_{\theta_1} \circ S_{\theta_2} \subseteq C_{\theta_1} \circ C_{\theta_2} $, by Lemma \ref{lemma: Fattore1} (the last inclusion is easily verified using the definitions), therefore $ C_{\theta_1} \circ C_{\theta_2}=\nabla_{\mathbf{I}}$. This shows that $(C_{\theta_1}, C_{\theta_2})$ is a pair of factor congruences on $\mathbf{I}$. \\

\noindent By Lemma \ref{lemma: int}-(2), for every $ i\in I $, $ (\theta_1)_{ii} \cap (\theta_2)_{ii}=(\theta_1 \cap \theta_2)_{ii}=(\Delta_{\mathbf{A}})_{ii}=\Delta_{\A_i}$. Moreover, observe that $S_{\theta_1} \cap S_{\theta_2} \subseteq C_{\theta_1} \cap C_{\theta_2} = C_{\theta_1 \cap \theta_2} = \Delta_{\I}$, that is $S_{\theta_1} \cap S_{\theta_2} = \Delta_{\I}$. Therefore, by Lemma \ref{lemma: Fattore1}, $(\theta_1)_{ii} \circ (\theta_2)_{ii}=(\theta_1 \circ \theta_2)_{ii} = (\theta_2 \circ \theta_1)_{ii}=(\theta_2)_{ii} \circ (\theta_{1})_{ii}$, for every $ i \in I $,  therefore $((\theta_1)_{ii},(\theta_2)_{ii})$ is a pair of permutable congruences. Consequently, for every $ i\in I: (\theta_1)_{ii} \vee (\theta_2)_{ii}=(\theta_1)_{ii} \circ (\theta_2)_{ii}=(\theta_1 \circ \theta_2)_{ii}=(\nabla_{\A})_{ii} = \nabla_{\A_i}$. This shows that for all $i \in I$, $((\theta_1)_{ii}, (\theta_2)_{ii})$ is a pair of factor congruences on $\mathbf{A}_i$.

    \noindent (2)$\Rightarrow$(1). Suppose that $\theta_{1}$ and $\theta_{2}$ are permutable and for every $i\in I$, $(C_{\theta_1}, C_{\theta_2})$ and $((\theta_1)_{ii}, (\theta_2)_{ii})$ are pairs of factor congruences on $\I$ and $\A_i$, respectively. 
    Since $S_{\theta_1} \cap S_{\theta_2} \subseteq C_{\theta_1} \cap C_{\theta_2} = C_{\theta_1 \cap \theta_2} = \Delta_{\I}$ then 
     $S_{\theta_1} \cap S_{\theta_2}=\Delta_{\mathbf{I}}$, and, by Lemma \ref{lemma: int}, $ (\theta_1 \cap \theta_2)_{ii}=(\theta_1)_{ii} \cap (\theta_2)_{ii}=\Delta_{\mathbf{A}_i}$, for every $ i\in I$. Therefore, by Theorem \ref{theorem: congruence5}, we deduce $\theta_1 \cap \theta_2 = \Delta_{\mathbf{A}}$. From the fact that $C_{\theta_1} $ and $ C_{\theta_2} $ are factor congruences (and Lemma \ref{lemma: Fattore1}), we have that $\nabla_{\mathbf{I}} = C_{\theta_1} \vee C_{\theta_2} = C_{\theta_1} \circ C_{\theta_2} \subseteq C_{\theta_1 \circ \theta_2} $, therefore $ C_{\theta_1 \circ \theta_2} = \nabla_{\I}$ and $\forall i\in I: (\theta_1 \vee \theta_2)_{ii} = (\theta_1 \circ \theta_2)_{ii}=(\theta_{1})_{ii} \circ (\theta_2)_{ii} = \mathbf{A}_i \times \mathbf{A}_i$, where we once again used Lemma \ref{lemma: Fattore1} considering the fact that $S_{\theta_1} \cap S_{\theta_2}=\Delta_{\mathbf{I}}$, thus $\forall (i,j) \in I \times I: (p_{ii\vee j} \times p_{ji\vee j})^{-1}((\theta_1 \circ \theta_2)_{i\vee j, i \vee j})=(p_{ii\vee j} \times p_{ji\vee j})^{-1}(A_{i \vee j} \times A_{i \vee j})=A_i \times A_j \neq \emptyset$, so $S_{\theta_1 \circ \theta_2}=C_{\theta_1 \circ \theta_2}=\nabla_{\mathbf{I}}$, and by the characterization of congruences in Theorem \ref{theorem: congruence5} we have $\theta_1 \vee \theta_2 = \nabla_{\A}$. This shows that $(\theta_1, \theta_2)$ is a pair of factor congruences on $\A$.
\end{proof}


\begin{remark}
Observe that the requiring $\theta_{1}$, $\theta_2$ to be permutable (congruences of the P\l onka sum $\A$) is a necessary assumption in (2) of Theorem \ref{th: congruenze fattore}, i.e. it is possible to find pairs $C_{\theta_{1}}, C_{\theta_2}$ and $(\theta_{1})_{ii}$, $(\theta_{2})_{ii}$ that are factor congruences on $\I$ and each fiber $\A_{i}$ of a P\l onka sum such that the corresponding congruences $\theta_{1}$, $\theta_{2}$ on the P\l onka sum $\A$ (via Theorem \ref{theorem: congruence5}) are \emph{not permutable}, as shown by the following example.\\
\noindent
Let $\A$ the involutive bisemilattice whose P\l onka sum is indexed over the (join-semilattice) diamond $I = \{0, i,j, 1\}$ and whose Boolean components are $B_{0} \cong B_j \cong B_1$ (copies of) the four-element Boolean algebra and $B_i$ (a copy of) the two-elements Boolean algebra, equipped with the boolean homomorphisms depicted in the following drawing. 
    
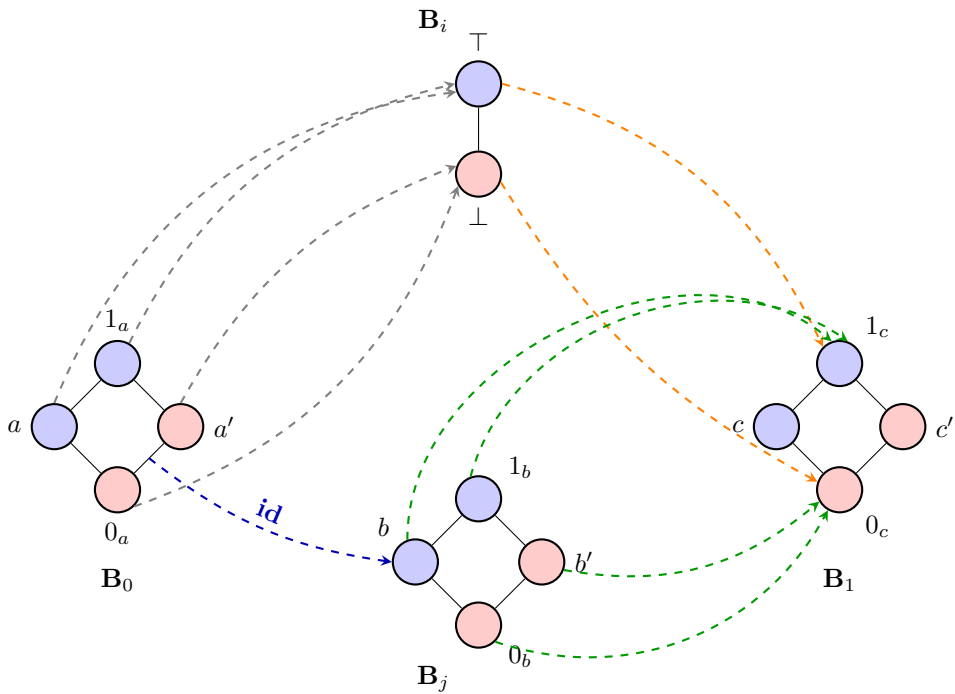
\begin{center}
\begin{tikzpicture}[scale=1.2, >=stealth]

  \tikzset{
    b4node/.style={circle, draw, thick, inner sep=0pt, minimum size=6mm},
    mappa/.style={->, dashed, thick}
  }

  \begin{scope}[shift={(0,0)}]
    \node[b4node, fill=blue!20, label=above:{$1_a$}] (S1)  at (0,1.4)  {};
    \node[b4node, fill=blue!20, label=left:{$a$}]   (Sa)  at (-0.7,0.7) {};
    \node[b4node, fill=red!20,  label=right:{$a'$}] (Sap) at (0.7,0.7)  {};
    \node[b4node, fill=red!20,  label=below:{$0_a$}]  (S0)  at (0,0)    {};
    \draw (S0)--(Sa)--(S1) (S0)--(Sap)--(S1);
    \node at (0,-1) {$\B_0$};
    
    \coordinate (idStart) at (0.35, 0.35);
  \end{scope}

  \begin{scope}[shift={(4,3.5)}]
    \node[b4node, fill=blue!20, label=above:{$\top$}] (L1) at (0,1) {};
    \node[b4node, fill=red!20,  label=below:{$\bot$}] (L0) at (0,0) {};
    \draw (L0)--(L1);
    \node at (-0.5,1.7) {$\B_i$};
  \end{scope}

  \begin{scope}[shift={(4,-1.5)}]
    \node[b4node, fill=blue!20, label=30:{$1_b$}]   (D1)  at (0,1.4)  {}; 
    \node[b4node, fill=blue!20, label=135:{$b$}]  (Da)  at (-0.7,0.7) {}; 
    \node[b4node, fill=red!20,  label=right:{$b'$}] (Dap) at (0.7,0.7)  {};
    \node[b4node, fill=red!20,  label=-30:{$0_b$}]  (D0)  at (0,0)    {}; 
    \draw (D0)--(Da)--(D1) (D0)--(Dap)--(D1);
    \node at (-0.5,-0.6) {$\B_j$};
  \end{scope}

  \begin{scope}[shift={(8,0)}]
    \node[b4node, fill=blue!20, label=45:{$1_c$}]   (T1)  at (0,1.4)  {}; 
    \node[b4node, fill=blue!20, label=left:{$c$}]   (Ta)  at (-0.7,0.7) {};
    \node[b4node, fill=red!20,  label=right:{$c'$}] (Tap) at (0.7,0.7)  {};
    \node[b4node, fill=red!20,  label=-45:{$0_c$}]  (T0)  at (0,0)    {}; 
    \draw (T0)--(Ta)--(T1) (T0)--(Tap)--(T1);
    \node at (0,-1) {$\B_1$};
  \end{scope}


  \begin{scope}[mappa, gray]
    \draw (S1.60) to[bend left=25] (L1.180);
    \draw (Sa.90) to[bend left=30] (L1.200);
    \draw (Sap.90) to[bend left=20] (L0.160);
    \draw (S0.-45) to[bend right=25] (L0.210);
  \end{scope}

  \draw[mappa, blue!70!black] (idStart) to[bend right=15] node[midway, sloped, above] {\textbf{id}} (Da.180);

  \begin{scope}[mappa, orange]
    \draw (L1.0) to[bend left=25] (T1.135);
    \draw (L0.-20) to[bend right=15] (T0.160);
  \end{scope}

  \begin{scope}[mappa, green!60!black]
    \draw (D1.110) to[bend left=55] (T1.70);
    \draw (Da.110) to[bend left=65] (T1.110);
    \draw (Dap.-20) to[bend right=25] (T0.-150);
    \draw (D0.-45) to[bend right=40] (T0.-120);
  \end{scope}

\end{tikzpicture}
\end{center}

Let $C_{\theta_{1}} = Cg^{\I}(i,1)$, $C_{\theta_{2}} = Cg^{\I}(1,j)$, $(\theta_{1})_{00} = Cg^{B_{0}}(a' , 1_{a})$, $(\theta_{1})_{jj} = Cg^{B_{j}}(b',1_{b})$, $(\theta_{1})_{ii} = \nabla_{{B_{i}}} $, $(\theta_{1})_{11}=\nabla_{B_{1}}$ and $(\theta_{1})_{00} = Cg^{B_{0}}(a , 1_{a})$, $(\theta_{1})_{jj} = Cg^{B_{j}}(b, 1_{b})$, $(\theta_{1})_{ii} = \Delta_{{B_{i}}} $, $(\theta_{1})_{11}=\Delta_{B_{1}}$. Then $C_{\theta_{1}}, C_{\theta_2}$ and $(\theta_{1})_{uu}$, $(\theta_{2})_{uu}$ (for every $u\in I$) are factor congruences on $\I$ and the fibers of $\A$, respectively, which moreover satisfies the conditions of Theorem \ref{theorem: congruence5}. The corresponding congruences (see Theorem \ref{theorem: congruence5}) $\theta_{1}$, $\theta_{2}$ on $\A$ are such that $[0_{a}]_{\theta_{1}}=\{0_{a}, 0_{b}, a, b\}$, $[1_{a}]_{\theta_{1}}=\{a', 1_{a}, b', 1_{b}\}$, $[0_{c}]_{\theta_{1}}=\{\bot, \top, 0_{c}, c, c', 1_{c}\}\}$ and $[0_{a}]_{\theta_{2}}=\{a',0_{a}, \bot\}$, $[1_{a}]_{\theta_{2}}=\{a, 1_{a}, \top\}$, $[1_{b}]_{\theta_{2}}=\{b, 1_{b}, 1_{c}\}$, $[0_{b}]_{\theta_{2}}=\{0_{b},0_{c}, b'\}$, $[c]_{\theta_{2}}=\{c\}$,$[c']_{\theta_{2}}=\{c'\}$
 and \emph{non permutable}, indeed $(c,b)\in \theta_{1}\circ\theta_{2}$ but $(c,b)\not\in \theta_{1}\circ\theta_{2}$.

\end{remark}

\section{Epimorphism surjectivity and Monomorphism injectivity}

\begin{definition}
    Let $\A$ and $\B$ be algebras in a class $\mathcal{K}$ and $f\in Hom(\A,\B)$ a homomorphism. $f$ is a $\mathcal{K}$-\emph{epimorphism} if it is right-cancellative, that is if for every algebra $\mathbf{C}$ in $\mathcal{K}$ and every $g_1,g_2 \in Hom(\B, \mathbf{C})$, $g_1\circ f=g_2\circ f$ implies $g_1=g_2$.
    $f$ is a $\mathcal{K}$-\emph{monomorphism} if it is left-cancellative, that is for every algebra $\mathbf{C}$ in $\mathcal{K}$ and every $g_1,g_2 \in Hom(\mathbf{C}, \A)$, $f\circ g_1=f\circ g_2$ implies $g_1=g_2$.
\end{definition}


\noindent 

Observe that every surjective (injective) homomorphism is an epimorphism (monomorphism), while the converse might not be the case. This justifies the following.

\begin{definition}
A class of algebras $\mathcal{K}$ has \emph{epimorphism surjectivity} (\emph{ES} for short) if each epimorphism of algebras in $\mathcal{K}$ is surjective. $\mathcal{K}$ has \emph{monomorphism injectivity} if every monomorphism of algebras in $\V$ is injective.
\end{definition}





\noindent Recall that an ideal of a (join) semilattice with zero $\Se$ is a subset $I\subseteq S$ that is a downset (i.e. for every $ x\in S $ and $ y\in I $ $ x\le y $ implies $x\in I$) and is $\vee$-closed, that is for every $ x,y\in I $, $ x\vee y \in I$. Moreover, a map $\varphi\colon\mathbf{S} \rightarrow \mathbf{2}$ is a (semilattice) homomorphism into the two-element semilattice $\mathbf{2}$ if and only if $\varphi^{-1}(\{0\})$ is an ideal of $\mathbf{S}$. Finally, recall that the variety of semilattices (with zero) has ES \cite{HornKimura}.  

The concept of epimorphism (monomorphism) can be defined for arbitrary categories in the obvious way. In the some of the following results we mention epimorphisms (monomorphisms) for the category of semilattice direct systems (with zero). 

\begin{lemma}\label{lemma: ES1}
    Let $h=(\varphi,(h_i)_{i\in I})\colon\mathbb{A}\rightarrow \mathbb{B}$ be an epimorphism of semilattice direct systems with zero in a strongly irregular variety $\V$. Then $\varphi\colon\mathbf{I} \rightarrow \mathbf{K}$ is an epimorphism.
\end{lemma}
\begin{proof}
    Suppose by contradiction that $\varphi\colon\mathbf{I} \rightarrow \mathbf{K}$ is not an epimorphism, that is $\varphi$ is not surjective. Then, using the technique in (the proof of) \cite[Theorem 1.1]{HornKimura}, consider two subsets of $K$ defined as follows: 
    \begin{itemize}
    \item $J_1=\downarrow a=\{x\in K\;:\; x \le a\}$; 
    \item $J_2=\{x\in K\;:\; \exists z\in \varphi(I) \text{ such that } x \le z < a\}$.
    \end{itemize} 
It is readily checked that $J_{1}$, $J_{2}$ are ideals such that $J_{1}\neq J_{2}$ and $J_{1}\cap\varphi(I) =J_{2}\cap\varphi(I)$. This implies that we can define two homomorphisms $\psi_1, \psi_2\colon\mathbf{K} \rightarrow \mathbf{2}$ such that $\psi_1 \neq \psi_2$ and $\psi_1\circ\varphi=\psi_2\circ\varphi$ (by setting $\psi_{i}(\{0\}) = J_{i}$, with $i\in\{1,2\}$). Given a semilattice homomorphism $\psi:\mathbf{K} \rightarrow \mathbf{2}$, consider the following between semilattice direct systems $\psi^{\star}:\mathbb{B} \rightarrow \mathbb{S}_2$, where $\mathbb{S}_2$ is the semilattice direct system formed by two trival algebras in $\V$ (and indexed over $\mathbf{2}$), defined as $\forall k\in K, \forall a\in \A_k: \psi^{\star}(a)=\psi(k)$. 
By construction we have that $\psi_1^{\star}\circ h=\psi_2^{\star}\circ h$, but $\psi_1^{\star} \neq \psi_2^{\star}$, in contrast with the hypothesis that $h$ is an epimorphism of semilattice direct systems with zero. 
\end{proof}

\begin{lemma}\label{lemma: ES2}
    Let $\A$ be a P\l onka sum over a semilattice direct system $\mathbb{A}=((A_i)_{i\in I}, (I, \le), (p_{ij})_{i\le j})$ in a strongly irregular variety $\V$ and let $g\in Hom(\A_i, \B)$, with $\B \in \V$, for some $i\in I$. Then $g $ can be extended to a homomorphism $h_g \in Hom(\A, \B^{\ast})$. 
\end{lemma}
\begin{proof}
Let $i\in I$ and $h_g\colon A\to B^{\ast}$ be defined as follows: 
\begin{equation}\label{eq: def h_g}
a \mapsto h_{g}(a):= \begin{cases}
    g\circ p_{ji}(a), \text{ if } a\in A_{j} \text{ with } j\leq i;\\
    \infty, \text{ otherwise}.
\end{cases}
\end{equation}
The map $h_g$ is clearly well defined and an extension of $g$ (as $p_{ii}= id_{\A_i}$). We only have to check that $h_g$ is a homomorphism, i.e. that $h_g(f^{\A}(a_{1},\dots,a_{n})) = f^{\B^{\ast}}(h_g(a_{1}),\dots, h_g(a_{n}))$, for any $n$-ary operation $f$ (in the type of $R(\V)$). To this end, it is enough to consider the following (exaustive) two cases: \\
1) $a_{j}\in A_{j}$ with $j\leq i$, for all $j\in\{1,\dots, n\}$. In this case $1\vee\dots\vee n = k\leq i$ (with a slight abuse of notation we are here using $1,\dots, n$ also as indexes for the algebras the elements $a_{1},\dots a_{n}$ belong to), therefore we have:
\begin{align*}
   h_g(f^{\A}(a_{1},\dots,a_{n})) &= h_g(f^{\A_k}(p_{1k}(a_{1}),\dots,p_{nk}(a_{n})))  \\
    &= g(p_{ki}(f^{\A_k}(p_{1k}(a_{1}),\dots,p_{nk}(a_{n})))) & \text{(definition of $h_g$)} \\
    &= g(f^{\A_i}(p_{ki}(p_{1k}(a_{1})),\dots,p_{ki}(p_{nk}(a_{n})))) & \text{($p_{ki}$ homomorphism)} \\
    & = g(f^{\A_i}(p_{1i}(a_{1}),\dots,p_{ni}(a_{n}))) \\
    &= f^{\B^{\ast}}((g(p_{1i}(a_{1}))), \dots ,g(p_{ni}(a_{n}))) & \text{($g$ homomorphism)}\\
    &= f^{\B^{\ast}}(h_g(a_{1}),\dots, h_g(a_{n}))) & \text{(definition of $h_g$)}
\end{align*}
2) $a_{j}\in A_{j}$ with $j\nleq i$, for some $j\in\{1,\dots, n\}$. Then $1\vee\dots\vee n = k\nleq i$ (as if $k\leq i$ then $j\nleq i$). Therefore $h_g(f^{\A}(a_{1},\dots,a_{n})) =h_g(f^{\A_k}(p_{1k}(a_{1}),\dots,p_{nk}(a_{n}))) = \infty  $ (as $k\nleq i$) and, on the other hand, $f^{\B^{\ast}}(h_g(a_{1}),\dots, h_g(a_{n})) = \infty $, as $h_g(a_{j}) = \infty$. 
\end{proof}



\begin{theorem}\label{th: ES per R(V)}
    Let $\V$ be a strongly irregular variety with epimorphism surjectivity, then its regularization $R(\V)$ has epimorphism surjectivity.
\end{theorem}

\begin{proof}
    Suppose that $\V$ has $ES$ and, by contradiction, that $R(\V)$ has not $ES$; that is there exist $\A, \B \in R(\V)$ and an epimorphism $g \in Hom(\A, \B)$ which is not surjective. We claim that there exist $\mathbf{C} \in \V$ and $h_1, h_2 \in Hom(\B, \mathbf{C}^*)$ such that $h_1 \neq h_2$ and $h_1 \circ g=h_2 \circ g$, which will be a contradiction to the fact that $g$ is by hypothesis an epimorphism. Since $g$ is not surjective, $\varphi_g$ is not surjective or there exists $i \in I$ such that $g_i: A_i \rightarrow B_{\varphi_g(i)}$ is not surjective. In the first case, the claim follows from the proof of Lemma \ref{lemma: ES1}. In the second case, $g_i:A_i \rightarrow B_{\varphi_g(i)}$ is not an epimorphism (since by hypothesis $\V$ has epimorphism surjectivity), so there exists $\mathbf{C} \in \V$ and $f_1, f_2 \in Hom(\B_{\varphi_g(i)}, \mathbf{C})$ such that $f_1 \circ g_i=f_2 \circ g_i$ and $f_1 \neq f_2$, while by Lemma \ref{lemma: ES2} we can extend these maps to homomorphisms $h_{f_1}, h_{f_2} \in Hom(\B, \mathbf{C}^*)$. It follows immediately from (\ref{eq: def h_g}) that $h_{f_1} \circ g=h_{f_2} \circ g$ and $h_{f_1} \neq h_{f_2}$. 
\end{proof}

\begin{example}
It follows from Theorem \ref{th: ES per R(V)} that the varieties of \emph{Clifford semigroups}, \emph{involutive bisemilattices} and \emph{dual weak braces} have $ES$. For the latter, observe that it is easy to show (using the fact that the variety of groups has ES) that the variety of \emph{skew braces} has ES, hence by Theorem \ref{th: ES per R(V)} also that of \emph{dual weak braces}.       
\end{example}

\noindent The proof strategy adopted in Theorem \ref{th: ES per R(V)} allows us to prove a stronger result, stated as follows.


\begin{theorem}\label{th: ES K}
    Let $\V$ be a strongly irregular variety with $ES$ and $\mathcal{K}$ a subclass of its regularization $R(\V)$ such that  for every $ \mathbf{C} \in \V$, there exists $ \mathbf{D} \in \mathcal{K}$ such that $\mathbf{C}^*  $ embeds in $ \mathbf{D}$. Then $\mathcal{K}$ has epimorphism surjectivity.
\end{theorem}

\begin{proof}
    Let $\A, \B \in \mathcal{K}$ and $h \in Hom(\A, \B)$ be a $\mathcal{K}$-epimorphism. We show that actually $h$ is also an $R(\V)$-epimorphism (to get the desired conclusion). Suppose, by contradiction, $h$ is not an $R(\V)$-epimorphism, then, applying the same argument in the proof of Theorem \ref{th: ES per R(V)} we have that there exist $\mathbf{C} \in \V$ and $g_1, g_2 \in Hom(\B, \mathbf{C}^*)$ such that $g_1 \neq g_2$ and $g_1 \circ h= g_2 \circ h$. By hypothesis there exists $\mathbf{D} \in \mathcal{K}$ such that $\mathbf{C}^*$ embeds in $\mathbf{D}$ via an injective homomorphism $\iota\colon \mathbf{C}^* \rightarrow \mathbf{D}$, therefore $(\iota \circ g_1) \circ h = (\iota \circ g_2) \circ h$. Therefore $\iota \circ g_1,\ \iota \circ g_2\colon \mathbf{B} \rightarrow \mathbf{D}^*$ are $\mathcal{K}$-homomorphisms and $h$ is, by assumption, a $\mathcal{K}$-epimorphism, so we deduce that $\iota \circ g_1 = \iota \circ g_2$, hence $g_1=g_2$ (since $\iota$ is injective), in contrast with the hypothesis $g_1 \neq g_2$. So $h$ is a $R(\V)$-epimorphism, therefore it is surjective, since $R(\V)$ has epimorphism surjectivity by Theorem \ref{th: ES per R(V)}.
\end{proof}



\noindent 

\noindent Recall that the two-element chain can be seen both as a semilattice and as a semilattice directed system (composed of two trivial algebras): in the following lemmas, we will denote the former by $\mathbf{2}$ and the latter by $\mathbbm{2}$. 

\begin{lemma}\label{lemma: MI1}
    Let $h=(\varphi,(h_i)_{i\in I})\colon \mathbb{A}\rightarrow \mathbb{B}$ be a monomorphism of semilattice direct systems with zero in a strongly irregular variety $\V$, then $\varphi  \colon\mathbf{I} \rightarrow \mathbf{K}$ is a monomorphism.
\end{lemma}

\begin{proof}
    Suppose by contradiction that $\varphi$ is not a monomorphism, that is $\varphi$ is not injective, i.e. there exist $i_1,i_2 \in I $ such that $i_1 \neq i_2 $ and $ \varphi(i_1)=\varphi(i_2)$. Consider the constant maps $\psi_1, \psi_2\colon \mathbf{2} \rightarrow \mathbf{I}$ onto $i_1$ and $i_2$, respectively, 
    which are semilattice homomorphisms, the semilattice direct system $\mathbbm{2}$ and the morphisms $k_1,k_2\colon \mathbbm{2} \rightarrow \mathbb{A}$ whose associated homomorphisms are $\psi_1$ and $\psi_2$, respectively. Then clearly $h\circ k_1=h\circ k_2$ but $k_1 \neq k_2$, in contradiction with the hypothesis that $h$ is a monomorphism.
\end{proof} 

\begin{lemma}\label{lemma: MI2}
    Let $\A$ and $\B$ be P\l onka sums over semilattice direct systems $((A_i)_{i\in I}, (I, \le), (p_{ij})_{i\le j})$ and $((B_h)_{h\in K}, (K, \le), (q_{hk})_{h\le k})$, respectively, in a strongly irregular variety $\V$ and $h=(\varphi,(h_i)_{i\in I})\in Hom(\A,\B)$ a monomorphism. Then $\forall i\in I: h_i:\A_i \rightarrow \B_{\varphi(i)}$ is a monomorphism.
\end{lemma}

\begin{proof}
    Suppose, by contradiction, that there exists $i\in I$ such that $h_i\colon\A_i \rightarrow \B_{\varphi(i)}$ is not a monomorphism, that is $h_i$ is not injective, since by assumption $\V$ has monomorphism injectivity. So we can find $a,b\in A_i$ such that $h_i(a)=h_i(b)$. Let $\psi\colon\mathbf{2} \rightarrow \mathbf{I}$ be the constant (semilattice) homomorphism onto $\{i\}$. Define the homomorphisms $k_1,k_2\colon\mathbbm{2} \rightarrow \mathbb{A}$ (whose associated semilattice homomorphism is $\psi$) as the constant maps onto $\{a\}$ and $\{b\}$, respectively. 
    It follows that $h\circ k_1=h\circ k_2$ but $k_1 \neq k_2$, in contrast with the hypothesis that $h$ is a monomorphism.
\end{proof}

\noindent The combination of Lemmas \ref{lemma: MI1} and \ref{lemma: MI2} yields immediately the following. 

\begin{theorem}
    Let $\V$ be a strongly irregular variety with monomorphism injectivity, then its regularization $R(\V)$ has monomorphism injectivity.
\end{theorem}


\section*{Ackowlegments}
We gratefully acknowledge the support of the following funding sources:
\begin{itemize}
    \item the University of Cagliari, through the StartUp project \emph{GraphNet} (F25F21002720001);
    \item the European Union (Horizon Europe Research and Innovation Programme), \\
    \noindent through the MSCA-RISE action MOSAIC (Grant agreement ID: 101007627);
    \item the Italian Ministry of Education, University and Research, through the PRIN 2022 project DeKLA (``Developing Kleene Logics and their Applications'', project code: 2022SM4XC8), the PRIN 2022 project "The variety of grounding" (project code: 2022NTCHYF), and the PRIN Pnrr project ``Quantum Models for Logic, Computation and Natural Processes (Qm4Np)'' (project code: P2022A52CR);
    \item the Fondazione di Sardegna, through the project MAPS (F73C23001550007);
    \item the ASL (Association for Symbolic Logic) for the Travel Grant issued to G. Zecchini for the participation to the LATD 2025 Conference in Siena; 
    \item the INDAM GNSAGA (Gruppo Nazionale per le Strutture Algebriche, Geometriche e loro Applicazioni).
\end{itemize} 

The results in the paper have been presented in various conferences and seminars. The authors express their gratitude in particular to the members of the ``general algebra'' group of the Warsaw University of Technology, and in particular to Anna Romanowska and Anna Zamojska-Dzienio. We also thank the members of the ALOPHIS group of the University of Cagliari and, in particular, Francesco Paoli.


\end{document}